\documentclass[11pt, twoside]{amsart}

\title[Frobenius monoidal functors induced by Frobenius extensions]{Frobenius monoidal functors induced by Frobenius extensions of Hopf algebras}
\date{March 30, 2026}

\author{Johannes Flake}
\address{Mathematical Institute, University of Bonn, Endenicher Allee 60, 53115 Bonn, Germany}
\email{flake@math.uni-bonn.de}
\urladdr{https://johannesflake.net}

\author{Robert Laugwitz}
\address{School of Mathematical Sciences,
University of Nottingham, University Park, Nottingham, NG7 2RD, UK}
\email{robert.laugwitz@nottingham.ac.uk}
\urladdr{https://www.nottingham.ac.uk/mathematics/people/robert.laugwitz}

\author{Sebastian Posur}
\address{University of Münster,
Fachbereich Mathematik und Informatik,
Einsteinstraße 62,
48149 Münster,
Germany}
\email{sebastian.posur@uni-muenster.de}

\usepackage{import}
\usepackage{amsmath,amsfonts,amsthm,amssymb}
\usepackage[alphabetic,abbrev,nobysame]{amsrefs}

\usepackage[english]{babel}

\usepackage{url}
\usepackage{fancyhdr}
\usepackage{graphicx}
\usepackage{microtype}
\usepackage{wrapfig, caption}
\usepackage[
colorlinks=true,
linkcolor=black, 
anchorcolor=black,
citecolor=black, 
urlcolor=black, 
]{hyperref}
\usepackage{cleveref} 
\usepackage{xcolor}
\usepackage{geometry}
\usepackage[all]{xy}
\usepackage{tikz}
\usepackage{dsfont}
\usetikzlibrary{calc}
\usepackage[shortlabels]{enumitem}
\usetikzlibrary{arrows}

\makeatletter
\newcommand{\superimpose}[2]{%
  {\ooalign{$#1\@firstoftwo#2$\cr\hfil$#1\@secondoftwo#2$\hfil\cr}}}
\makeatother

\newcommand{\leftexpsub}[3]{{\vphantom{#3}}^{#1}_{#2}{#3}}

\newcommand{\lYD}[1]{\leftexpsub{#1}{#1}{\mathbf{YD}}}

\newcommand{\oop}{\mathrm{op}}

\newcommand{\ov}[1]{\overline{#1}}
\newcommand{\un}[1]{\underline{#1}}

\newcommand{\lmod}[1]{#1\text{-}\mathbf{mod}}

\newcommand{\lMod}[1]{#1\text{-}\mathbf{Mod}}

\newcommand{\rModint}[2]{\mathbf{Mod}_{#1}\text{-}#2}
\newcommand{\rModloc}[2]{\mathbf{Mod}^{\mathrm{loc}}_{#1}\text{-}#2}

\newcommand{\projr}[2]{\mathrm{rproj}_{{#1},{#2}}}
\newcommand{\projl}[2]{\mathrm{lproj}_{{#1},{#2}}}
\newcommand{\projrnoarg}{\mathrm{rproj}}
\newcommand{\projlnoarg}{\mathrm{lproj}}

\newcommand{\iprojr}[2]{\mathrm{rproj}^{F,F\dashv G}_{{#1},{#2}}}
\newcommand{\iprojl}[2]{\mathrm{lproj}^{F,F\dashv G}_{{#1},{#2}}}

\newcommand{\Alg}{\mathbf{Alg}}

\newcommand{\cha}{\operatorname{char}}
\newcommand{\CoAlg}{\mathbf{CoAlg}}

\newcommand{\Drin}{\operatorname{Drin}}

\newcommand{\Frob}{\mathbf{Frob}}
\newcommand{\Hom}{\operatorname{Hom}}
\newcommand{\Ind}{\operatorname{Ind}}
\newcommand{\CoInd}{\operatorname{CoInd}}

\newcommand\id{{\operatorname{id}}}

\newcommand{\isomorph}{\stackrel{\sim}{\longrightarrow}}
\newcommand{\lax}{\operatorname{lax}}
\newcommand{\oplax}{\operatorname{oplax}}
\newcommand{\unit}{\operatorname{unit}}
\newcommand{\counit}{\operatorname{counit}}

\newcommand{\one}{\mathds{1}}

\newcommand{\Res}{\operatorname{Res}}

\newcommand{\tr}{\operatorname{tr}}

\newcommand{\Set}[1]{\left\lbrace #1\right\rbrace}
\newcommand{\inner}[1]{\left\langle #1\right\rangle}

\newcommand{\sfC}{\mathsf{C}}
\newcommand{\sfG}{\mathsf{G}}
\newcommand{\sfK}{\mathsf{K}}

\newcommand{\sfN}{\mathsf{N}}
\newcommand{\sfH}{\mathsf{H}}

\newcommand{\OH}{\mathcal{O}_\mathsf{H}}

\providecommand{\fr}[1]{\mathfrak{#1}}

\newcommand{\mC}{\mathbb{C}}

\newcommand{\mZ}{\mathbb{Z}}

\newcommand{\cC}{\mathcal{C}}
\newcommand{\cD}{\mathcal{D}}

\newcommand{\cZ}{\mathcal{Z}}

\newtheoremstyle{mystyle}
  {0.5cm}                   
  {0.5cm}                   
  {\normalfont}           
  {}                      
  {\itfont\bfseries} 
  {:}                     
  {0.3cm}              
  {\thmname{#1}}

\newtheoremstyle{defstyle}
  {0.5cm}                   
  {0.5cm}                   
  {\normalfont}           
  {}     
  {\normalfont\bfseries}  
  {:}                     
  {0.3cm}              
  {\thmname{#1}\thmnumber{ #2}\thmnote{ (#3)}}

\newtheorem*{rep@theorem}{\rep@title}
\newcommand{\newreptheorem}[2]{%
\newenvironment{rep#1}[1]{%
 \def\rep@title{#2 \ref{##1}}%
 \begin{rep@theorem}}%
 {\end{rep@theorem}}}
\makeatother

\newtheorem{theorem}{Theorem}[section]

\newtheorem{proposition}[theorem]{Proposition}
\newreptheorem{proposition}{Proposition}
\newtheorem{corollary}[theorem]{Corollary}
\newreptheorem{corollary}{Corollary}
\newtheorem{lemma}[theorem]{Lemma}

\newtheorem*{theorem*}{Theorem}
\newreptheorem{theorem}{Theorem}

\newtheorem{introtheorem}{Theorem}

\newtheorem{introcorollary}[introtheorem]{Corollary}

\theoremstyle{definition}
\newtheorem{definition}[theorem]{Definition}

\theoremstyle{remark}
\newtheorem{example}[theorem]{Example}
\newtheorem{remark}[theorem]{Remark}

\newtheorem{question}[theorem]{Question}

\numberwithin{equation}{section}

\tikzset{mylabel/.style={fill=white,font=\small},font=\small}


\begin{document}

\begin{abstract}
We show that induction along a Frobenius extension of Hopf algebras is a Frobenius monoidal functor in great generality, in particular, for all finite-dimensional and all pointed Hopf algebras. As an application, we show that induction functors from unimodular Hopf subalgebras to small quantum groups at roots of unity are Frobenius monoidal functors and classify such unimodular Hopf subalgebras. 
Moreover, we present stronger conditions on Frobenius extensions under which the induction functor extends to a \emph{braided} Frobenius monoidal functor on categories of Yetter--Drinfeld modules. We show that these stronger conditions hold for any extension of finite-dimensional semisimple and co-semisimple (or, more generally, unimodular and dual unimodular) Hopf algebras.
\end{abstract}

\maketitle

\section{Introduction}

Let $K\subset H$ be an extension of Hopf algebras over a field $\Bbbk$ and consider the strong monoidal restriction functor 
$$\Res_K^H\colon \lMod{H}\to \lMod{K}.$$
This functor has a left and a right adjoint, given by induction and coinduction
$$\Ind_K^H\colon \lMod{K}\to \lMod{H}, \qquad \CoInd_K^H\colon \lMod{K}\to \lMod{H}.$$
The extension $K\subset H$ is a \emph{Frobenius extension} if there is a natural isomorphism $\Ind_K^H\cong \CoInd_K^H$. In this case, $\Ind_K^H$ is a  \emph{Frobenius functor} \cites{CMZ,ShiU}, i.e.\,a simultaneous left and right adjoint to $\Res_K^H$. 
By \cite{KelDoc}, the functor $\Ind_K^H$ then has both a lax and an oplax monoidal structure
\begin{gather*}
    \lax_{V,U}\colon \Ind_{K}^H(V)\otimes \Ind_{K}^H(U)\to \Ind_{K}^H(V\otimes U),\\
    \oplax_{V,U}\colon \Ind_{K}^H(V\otimes U)\to \Ind_{K}^H(V)\otimes \Ind_{K}^H(U).
\end{gather*}
It is a natural question to ask what compatibilities these structures satisfy. It turns out that $\Ind_K^H$ is almost never a strong monoidal functor, where the lax and oplax structures are inverse to each other, but often a \emph{Frobenius monoidal functor}, where they satisfy a compatibility similar to the axioms of a Frobenius algebra.

\begin{introtheorem}[See \Cref{thm:Hopf-Frobenius-monoidal}, \Cref{thm-Ind-Frob-mon}]\label{intro-thm:A}
Assume that $H$ is either finite-dimensional or pointed.\footnote{Or assuming the more general condition that $K\subset H$ is an extension of \emph{right integral type}, \Cref{def:integral-type} or \cite{FMS}*{1.10.~Definition}.} If $K\subset H$ is a Frobenius extension, then $\Ind_K^H$ is a Frobenius monoidal functor.
\end{introtheorem}

The study of Frobenius extensions of rings is by now a classical topic that goes back to \cites{NT,Kasch}. It was continued in the context of Hopf algebras by \cites{Sch,Farn,FMS}. The above result \Cref{intro-thm:A} builds on the work of \cites{Sch,FMS} using our results in \cite{FLP3}.  It contributes to the study of conditions on right adjoints of strong monoidal functors to give Frobenius monoidal functors \cites{Bal, Yad, FLP3}. 

\smallskip

There are several special cases when the conclusions of Theorem \ref{intro-thm:A} holds.

\begin{introcorollary}[See \Cref{cor:Ind-Frob-conditions} and \Cref{cor:Ind-Frob-conditions-fd}]
Assume that $H$ is finitely generated and faithfully flat over $K$ and that either one of the following conditions holds:
    \begin{enumerate}[(i)]
        \item $K$ is a normal Hopf subalgebra of $H$,
        \item $K\subset Z(H)$,
        \item $H$ is commutative.
\end{enumerate}
 Or, assume that $H$ is finite-dimensional, and one of the conditions (i)--(vii) holds, where:
\begin{enumerate}
        \item[(iv)] $\alpha_H(k)=\alpha_K(k)$, for all $k\in K$, where $\alpha_H$ and $\alpha_K$ are the distinguished grouplike elements of $H^*$ and $K^*$, respectively,
        \item[(v)] both $H$ and $K$ are unimodular,
        \item[(vi)] $H$ is semisimple.
    \end{enumerate}
Then $\Ind_K^H$ is a Frobenius monoidal functor.
\end{introcorollary}

Building on our work in \cites{FLP2,FLP3} and the theory of integrals of Hopf algebras, the second main result of this paper gives conditions when $\Ind_K^H$ lifts to a braided Frobenius monoidal functor on the \emph{Drinfeld centers} \cites{JS,Maj2} of the categories of modules over the Hopf algebras $K$ and $H$. Here, the Drinfeld center $\cZ(\lMod{H})$ is, equivalently, realized as the category of \emph{Yetter--Drinfeld modules} $\lYD{H}$ \cite{Yet}.

\begin{introtheorem}[See \Cref{thm:H*-unimodular}]\label{thm:B}
    Let $K\subset H$ be a Frobenius extension of finite-dimensional Hopf algebras. If the dual Hopf algebras $K^*$ and $H^*$ are unimodular, then  the functor
    \begin{gather*}
    \cZ(\Ind_K^H)\colon \lYD{K}\to \lYD{H},\quad (V,\delta)\mapsto (\Ind_K^H(V),\delta^{\Ind_K^H(V)}),\\
    \delta^{\Ind_K^H(V)}(h\otimes v)=h_{(1)}v^{(-1)}S(h_{(3)})\otimes h_{(2)}\otimes v^{(0)},\qquad v\in V, h\in H,
\end{gather*}
from \cite{FLP2}*{Corollary~7.2}  is a braided Frobenius monoidal functor.
\end{introtheorem}
An important special case of \Cref{thm:B} concerns (co)semisimple Hopf algebras. 

\begin{introcorollary}[See \Cref{cor:H-semisimple}]
Assume that $K\subset H$ is an extension of finite-dimensional Hopf algebras such that either
\begin{enumerate}
    \item[(i)] $H$ is semisimple and co-semisimple, or
    \item[(ii)] $H$ is semisimple and $\cha \Bbbk =0$,
\end{enumerate}
Then $\cZ(\Ind_K^H)$ is a braided Frobenius monoidal functor.
\end{introcorollary}

We apply our results to various classes of examples of Frobenius extensions of Hopf algebras, including:
\begin{itemize}
\item Some examples of restricted Lie algebras, see \Cref{ex:restrictedLie}.
\item Induction along the extension $H\subset \Drin(H)$ for the \emph{Drinfeld double} $\Drin(H)$ of a unimodular Hopf algebra $H$ gives a Frobenius monoidal functor, see \Cref{sec:Drin-double}.
\item The inclusion of the large central subalgebra $\OH\subset U_\epsilon(\fr{g})$ into the \emph{De Concini--Kac--Procesi form of the quantized enveloping algebra} at a root of unity $\epsilon$ of odd order $\ell$ gives a Frobenius monoidal induction functor. This functor only extends to Yetter--Drinfeld categories as a braided lax or oplax monoidal functor but not as a braided Frobenius monoidal functor, see \Cref{sec:quantum-groups}.
\item If $\ell>3$ we show that any Hopf subalgebra $K\subset u_\epsilon(\fr{g})$ of the \emph{small quantum group} that yields a Frobenius extension is an extension of $u_\epsilon(\fr{h})$, for a Lie subalgebra $\fr{h}\subset\fr{g}$, by an enlargement of the Cartan part, see \Cref{sec:u-subHopf}.
\item Finally, we consider the \emph{Kac--Paljutkin algebra} $H_8$, an $8$-dimensional semisimple Hopf algebra, which contains the group algebra of the Klein four-group $\sfC_2\times \sfC_2$. We investigate the extension to Drinfeld centers of the induction functor along this Frobenius extension in detail and identify the simple objects in its image. Moreover, we use this functor to construct rigid Frobenius algebras in $\cZ(\lMod{H_8})$, see Sections~\ref{sec:KP-algebras}--\ref{sec:Frob-alg-KP}.
\end{itemize}

It is an interesting question which of the results of the present paper can be generalized to (finite abelian) tensor categories.
After completion of the present paper, \cite{JY} appeared which contains several results in this direction. In particular, \Cref{question1} and \Cref{question2} were answered in the affirmative in \cite{JY}*{Theorem 3.23}. We thank Harshit Yadav and Kenichi Shimizu for helpful comments on topics related to this work.\footnote{We also thank the anonymous referee for helping us improve this manuscripts, see especially \Cref{ex:small-Frob-big}.}

The paper is structured as follows: We start by including necessary preliminaries in \Cref{sec:prelim} and present results on Frobenius extensions of Hopf algebras, mostly obtained from \cites{Sch, FMS}, in \Cref{sec:Hopf}. The main results of the paper are found in \Cref{sec:main} while \Cref{sec:examples} is devoted to classes of examples. At the end, in \Cref{sec:open-questions}, we include some open questions.

\setcounter{tocdepth}{3}

\section{Preliminaries}\label{sec:prelim}

We briefly recall the concepts required in this paper and fix notation. Throughout, $\Bbbk$ denotes a field and further assumptions on $\Bbbk$ will be stated when needed. We refer the reader to \cites{Mon,Maj,Rad} for textbook accounts of the theory of Hopf algebras and \cite{EGNO} for the theory of monoidal and braided monoidal categories. 

\subsection{Hopf algebras and their module categories}

Let $H$ be a Hopf algebra over $\Bbbk$, by which we mean a bialgebra with product $m_H$, unit $1_H\in H$, coproduct $\Delta_H\colon  H\to H\otimes H$, and counit $\varepsilon_H\colon H\to \Bbbk$, that has an invertible antipode $S_H\colon H\to H$. We use adapted Sweedler's notation, with suppressed summation,
$$\Delta_H(h)=h_{(1)}\otimes h_{(2)}\in H\otimes H,$$
to denote the coproduct evaluated at $h\in H$. We use the standard convention to denote 
$$(\Delta_H\otimes \id)\Delta_H(h)=h_{(1)}\otimes h_{(2)}\otimes h_{(3)}=(\id\otimes\Delta_H)\Delta_H(h),$$
which is justified by coassociativity of $\Delta_H$.
A morphism of Hopf algebras $\varphi\colon K \to H$ is required to be a morphism of both algebras and coalgebras that commutes with the antipodes, i.e.\,$\varphi S_K=S_H\varphi$. We denote the $\Bbbk$-dual of a finite-dimensional Hopf algebra $H$ by $H^*$.

We will consider the monoidal category $\lMod{H}$ of left $H$-modules. The monoidal subcategory of finite-dimensional $H$-modules is denoted by $\lmod{H}$. On the tensor product $V\otimes W=V\otimes_\Bbbk W$ of two $H$-modules $V$ and $W$, $H$ acts via 
$h\cdot (v\otimes w)=h_{(1)}\cdot v\otimes h_{(2)}\cdot w$ and the tensor unit $\one$ is given by $\Bbbk$ with $H$-action via $h\cdot 1=\varepsilon(h)$, for $h\in H$.

\subsection{Induction and restriction functors}

Given a morphism of Hopf algebras $\varphi\colon K\to H$, consider the restriction functor
$$
G=\Res_\varphi\colon \lMod{H}\to \lMod{K}, \quad W\mapsto \left.W\right|_K.
$$
Restriction has a left and a right adjoint functor given by
\begin{align}
  L=\Ind_\varphi(V):=H\otimes_K V, \qquad R=\CoInd_\varphi(V)=\Hom_{K}(H,V),
\end{align}
for a $K$-module $V$. The left $H$-module structure on $\CoInd_\varphi(V)$ is
\begin{align}
    h\cdot f(g)=f(gh), \qquad f\in \Hom_{K}(H,V), ~~ h,g\in H.
\end{align}

In this paper, we will focus on the case of an inclusion $\iota\colon K\hookrightarrow H$ of Hopf algebras and denote the induction and restriction functors by 
$$\Ind_K^H=\Ind_\iota\quad\dashv\quad  \Res_K^H=\Res_\iota \quad\dashv\quad \CoInd_K^H=\CoInd_\iota.$$

\subsection{Lax, oplax, and Frobenius monoidal functors}

We will briefly recall the definitions of lax, oplax, strong and Frobenius monoidal functor between monoidal categories  $\cC$ and $\cD$. Let $\cC, \cD$ denote monoidal categories which, by Mac Lane's coherence theorem \cite{ML}, we can regard as strict monoidal for simplicity. A \emph{lax monoidal functor} is a functor $R: \cD \rightarrow \cC$ together with morphisms
\begin{equation}
    \lax_{X,Y}\colon RX \otimes RY \longrightarrow R(X \otimes Y),
\end{equation}
natural in $X,Y \in \cD$ and a morphism $\lax_{0}\colon \one_\cC\rightarrow R(\one_\cD)$
such that the associativity constraint 
\begin{equation}\label{eq:associativity}
   \lax_{X \otimes Y, Z} (\lax_{X, Y} \otimes \id)= \lax_{X, Y \otimes Z}(\id \otimes \lax_{Y,Z})
\end{equation}
holds for $X,Y,Z \in \cD$ as well as the  unitality constraints 
\begin{align}\label{eq:unitality}
    \lax_{\one,X}(\lax_0\otimes \id)=\id_{RX}, \qquad \lax_{X,\one}(\id\otimes \lax_0)=\id_{RX}.
\end{align}
Dually, one defines an \emph{oplax monoidal functor} as a functor $L: \cD \rightarrow \cC$ with morphisms
\begin{equation}
        \oplax_{X,Y}\colon L(X \otimes Y) \longrightarrow LX \otimes LY,
\end{equation}
natural in $X,Y$ and a morphism $\oplax_0\colon L(\one)\to \one$ satisfying the conditions obtained from \Cref{eq:associativity} and \Cref{eq:unitality} by duality.

A \emph{strong monoidal functor} can now be defined as a lax monoidal functor where $\lax_{X,Y}$ and $\lax_0$ are isomorphisms. 

The following result is due to \cite{KelDoc}, see also \cite{FLP2}*{Proposition~2.2} for details.
\begin{lemma}
    If $G$ is strong monoidal, $L$ left adjoint to $G$, and $R$ right adjoint to $G$. Then $L$ is oplax monoidal and $R$ is lax monoidal.
\end{lemma}
This lemma can be applied to the Hopf algebra case.  
\begin{example}
    For a morphism of Hopf algebras $\varphi\colon K\to H$, we note that $G=\Res_\varphi$ is a strong monoidal functor and will identify $\left. (W\otimes X)\right|_K$ and $\left.W\right|_K\otimes \left. X\right|_K$ for $H$-modules $W,X$ to simplify notation. Then $R=\CoInd_\varphi$ and $L=\Ind_\varphi$ have
the following lax and oplax monoidal structures, for $K$-modules $V$ and $U$:
\begin{gather}
    \lax^R_{V,U}\colon R(V)\otimes R(U)\longrightarrow R(V\otimes U), \quad f\otimes g\mapsto \big(h\mapsto f(h_{(1)})\otimes g(h_{(2)})\big),\label{eq:lax-Hopf}\\
    \oplax^{L}_{V,U}\colon  L(V\otimes U)\to L(V)\otimes L(U), \quad h\otimes_K (v\otimes u)\mapsto (h_{(1)}\otimes_K v)\otimes (h_{(2)}\otimes_K u).\label{eq:oplax-Hopf}
\end{gather}
\end{example}

Finally, a \emph{Frobenius monoidal functor} (see e.g.\,\cite{DP}) $F\colon \cD\to \cC$ is a functor equipped with a lax monoidal structure $(\lax,\lax_0)$, and an oplax monoidal structure $(\oplax, \oplax_0)$ satisfying the following two compatibility conditions, for any objects $X,Y,Z$ of $\cD$, 
\begin{align}
  (\lax_{X,Y}\otimes \id_{F(Z)})(\id_{F(X)}\otimes \oplax_{Y,Z}) & =  \oplax_{X\otimes Y,Z}\lax_{X,Y\otimes Z}\label{frobmon1}\\
  (\id_{F(X)}\otimes \lax_{Y,Z})(\oplax_{X,Y}\otimes \id_{F(Z)}) & = \oplax_{X,Y\otimes Z}\lax_{X\otimes Y,Z}. \label{frobmon2}
\end{align}

Recall that $\cC$ is \emph{braided monoidal} if it comes with natural isomorphisms 
$\Psi^\cC_{X,Y}\colon X\otimes Y\to Y\otimes X$, called the \emph{braiding} which satisfy the hexagon axioms (see e.g.\,\cite{EGNO}*{Section~8.1}). For  $\cC,\cD$ braided monoidal categories,  say that a lax (or, oplax) monoidal functor $F\colon \cD\to \cC$ is \emph{lax braided} (respectively, \emph{oplax braided}) if the conditions
\begin{align}
    \lax_{Y,X}\Psi^\cC_{F(X),F(Y)}&=F(\Psi^\cD_{X,Y})\lax_{X,Y}, \label{eq:Frob-braided-lax}\\
        \Psi^\cC_{F(X),F(Y)}\oplax_{X,Y}&=\oplax_{Y,X}F(\Psi^\cD_{X,Y}),\label{eq:Frob-braided-oplax}
\end{align}
hold for all objects $X,Y\in \cD$.
A Frobenius monoidal functor between braided monoidal categories is called \emph{braided} if it is both lax braided and oplax braided.

\subsection{Algebraic structures in monoidal categories}\label{sec:alg-in-C}

Let $\cC$ be a monoidal category. We can define an \emph{algebra in $\cC$} as an object $A$ in $\cC$ together with morphisms $m_A\colon A\otimes A\to A$, $u_A\colon \one \to A$ in $\cC$ which satisfy the associativity and unitality axioms internally to $\cC$. Algebras in $\cC$ form a category denoted by $\Alg(\cC)$. Dually, we define coalgebras in $\cC$ which form a category $\CoAlg(\cC)$. One checks that a lax\slash oplax monoidal functor $F\colon \cD\to \cC$ induces a functor $\Alg(\cD)\to \Alg(\cC)$\slash a functor $\CoAlg(\cD)\to \CoAlg(\cC)$. 
The (co)multiplication and (co)units on the images of these functors are defined by 
\begin{gather}
    m_{F(A)}=F(m_A)\lax_{A,A}, \qquad u_{F(A)}=F(u_A)\lax_{0},\label{eq:F-alg}\\
    \Delta_{F(C)}=\oplax_{C,C}F(\Delta_C), \qquad \varepsilon_{F(C)}=\oplax_{0} F(\varepsilon_C),\label{eq:F-coalg}
\end{gather}
where $(A,m_A,u_A)$ is an algebra and $(C,\Delta_A,\varepsilon_A)$ is a coalgebra in $\cD$.

Recall that a \emph{Frobenius algebra} in a monoidal category $\cC$ is an object $A$ in $\cC$ which is an algebra $(A,m_A,u_A)$ and a coalgebra $(A,\Delta_A,\varepsilon_A)$ such that the compatibility conditions
\begin{gather}\label{eq:Frobaxiom}
(m_A\otimes \id_A)(\id_A\otimes \Delta_A)=\Delta_A m_A=(\id_A\otimes m_A)(\Delta_A\otimes \id_A)
\end{gather}
hold.
A \emph{morphism of Frobenius algebras} is required to commute with both the algebra and coalgebra structures. We denote the category of Frobenius algebras in $\cC$ by $\Frob(\cC)$. One checks that a Frobenius monoidal functor $F\colon \cD\to \cC$ induces a functor
$\Frob(\cD)\to\Frob(\cC)$ that combines the induced functors on algebras and coalgebras defined in Equations~\eqref{eq:F-alg}--\eqref{eq:F-coalg}.

Recall that an algebra $A$ in a braided monoidal category $\cC$ with braiding $\Psi$ is \emph{commutative} if $m_A\Psi_{A,A}=m_A$ and a coalgebra $C$ is \emph{cocommutative} if  $\Psi_{C,C}\Delta_C=\Delta_C$. A commutative Frobenius algebra is a Frobenius algebra which is commutative as an algebra and note that this implies cocommutativity as a coalgebra. 

One checks that a braided lax (respectively, braided oplax) monoidal functor preserves commutative algebras (respectively, cocommutative coalgebras). Hence, a braided Frobenius monoidal functor preserves commutative Frobenius algebras.

\subsection{The Drinfeld center, Yetter--Drinfeld modules, and the Drinfeld double}

For a monoidal category $\cC$, the \emph{Drinfeld center} $\cZ(\cC)$ is a braided monoidal category, see \cite{Maj2}*{Example~3.4} and \cite{JS}*{Definition~3}. 
In this paper, we will describe the Drinfeld center $\cZ(\lMod{H})$ by the equivalent category $\lYD{H}$ of \emph{Yetter--Drinfeld modules} (or \emph{YD modules} for short) over a Hopf algebra $H$, see \cite{Yet}, \cite{Maj}*{Example 9.1.8}.
This category $\lYD{H}$ consists of objects $V$ which are both left $H$-modules and left $H$-comodules with the coaction
$$\delta\colon V\to H\otimes V, \qquad v\mapsto v^{(-1)}\otimes v^{(0)},$$
satisfying the \emph{Yetter--Drinfeld condition}
\begin{equation}\label{eq:YD-cond2}
\delta(h v)=h_{(1)}v^{(-1)}S(h_{(3)})\otimes h_{(2)}v^{(0)},
\end{equation}
for $h\in H, v\in V$. 
The tensor product $V\otimes W$ of two YD modules is simply the tensor product of the underlying $H$-modules and $H$-comodules via $\Delta_H$ and $m_H$, respectively. 
The braiding $\Psi_{V,W}\colon V\otimes W\to W\otimes V$ of two YD modules is given by 
\begin{equation}\label{eq:YDbraiding}
    \Psi_{V,W} \colon V\otimes W\to W\otimes V, \qquad v\otimes w \mapsto (v^{(-1)}\cdot w)\otimes v^{(0)}, 
\end{equation}
for any $v\in V,w\in W$.
It shall be noted that, when $H$ is finite-dimensional, the category of Yetter--Drinfeld modules is equivalent to modules over a Hopf algebra, called the \emph{Drinfeld double} of $H$, see e.g.\,\cite{EGNO}*{Section~7.14}. In this paper, we will work with the YD module description of the center which is more convenient, especially when working with infinite-dimensional Hopf algebras.

\section{Frobenius extensions of Hopf algebras}
\label{sec:Hopf}

We now discuss results on Frobenius extensions of Hopf algebras following \cites{Sch,FMS}. We start by describing isomorphisms of induction and coinduction functors in \Cref{sec:Ind-Coind-iso}. The associated augmented algebra is introduced as a key tool to study Frobenius extensions in \Cref{sec:ovH*}. The concepts of extensions of right integral type and normal bases, and criteria when such extensions are Frobenius extensions, are recalled in Sections \ref{sec:FrobExt}--\ref{sec:FrobExt3}. 

\subsection{Isomorphisms of induction and coinduction functors}
\label{sec:Ind-Coind-iso}

Let $K\subset H$ be an inclusion of Hopf algebras (from now on called an \emph{extension} of Hopf algebras).

\begin{remark}
We say that $K\subset H$ is a \emph{finite extension} if $H$ is finitely generated as a left $K$-module.
We note that, as we only consider Hopf algebras with invertible antipode, $H$ is finitely generated\slash projective\slash faithfully flat\slash free as a left $K$-module if and only if $H$ has the respective property as a right $K$-module. Thus, we will say that \emph{$H$ is  finite\slash projective\slash faithfully flat\slash free over $K$} in these cases without ambiguity.
\end{remark}

We first recall when
induction and coinduction along $K\subset H$ are isomorphic based on maps called \emph{Frobenius morphisms}, see \cites{Kas,NT,MN}. 

\begin{proposition}\label{thm:ind-coind}
The following statements are equivalent.
\begin{enumerate}
    \item[(i)] The functors $\Ind_K^H$ and  $\CoInd_K^H$ are isomorphic.
    \item[(ii)]
       \begin{enumerate}
        \item[(a)] $H$ is finite projective over $K$, and 
        \item[(b)] there is a morphism of $K$-$K$-bimodules
        $\tr \colon H \to K$
which induces a bijection $\theta_H \colon H\to \Hom_K(H,K), h\mapsto \tr((-)h)$.
    \end{enumerate}
\end{enumerate}
\end{proposition}

We adopt the following terminology from \cites{Sch,FMS}. 
\begin{definition}[Frobenius extensions and Frobenius morphisms]\label{def:Frob-mor}
     A map $\tr$ satisfying the conditions in \Cref{thm:ind-coind}(ii)(b) is called a \emph{Frobenius morphism}. If such a Frobenius morphism exists for $H$ finite projective over $K$, then we call $K\subset H$ a \emph{Frobenius extension} of Hopf algebras.
\end{definition}

If $H$ is finite projective over $K$, Frobenius morphisms are in bijection with natural isomorphisms $\theta\colon \Ind_K^H\isomorph \CoInd_K^H$.
Indeed, a Frobenius morphism $\tr\colon H\to K$ 
gives a natural isomorphism $\theta\colon \Ind_K^H\to \CoInd_K^H$ with component maps
\begin{align}\label{lem:thetaV}
\theta_V(h\otimes v)=\big(g\mapsto \tr(gh)v \big).
\end{align}

\begin{example}\label{ex:groups}
Let $\mathsf{G}$ be a group and $\sfK\subset \sfG$ a subgroup of finite index. One shows as in \cite{FLP3}*{Example~5.22} that $\tr(g)=g$, if $g\in \sfK$, and $\tr(g)=0$ otherwise
extends to a Frobenius morphism $\tr\colon \Bbbk \sfG\to \Bbbk \sfK$.
\end{example}

Not all extensions of Hopf algebras admit a Frobenius morphism, for instance, the inclusion of the group algebra $\Bbbk \sfC_\ell$ of a cyclic group of order $\ell>2$ into the Taft algebra $T_\ell(q)$, associated to an $\ell$-th root of unity, does not, as shown in \cite{FLP3}*{Example~5.7}.

\smallskip

For Frobenius extensions of Hopf algebras, it is useful to fix the following dual sets of generators for $H$, see \cite{FLP3}*{Lemma~5.8} and cf.\,\cite{FMS}*{1.3.\,Proposition}.

\begin{corollary}\label{lem:kideltai}
If $K\subset H$ is a Frobenius extension with Frobenius morphism $\tr$, then there exist $n\geq 1$ and generators $h_1,\ldots,h_n$ of $H$ as a left $K$-module and $\delta_1,\ldots, \delta_n$ of $H$ as a right $K$-module such that, for all $h\in H$,
\begin{align}
\sum_{i=1}^n \delta_i\otimes_K h_i h &= \sum_{i=1}^n h\delta_i\otimes_K h_i,\label{eq:kideltai1} \\
h =\sum_{i=1}^n \tr(h\delta_i)h_i&= \sum_{i=1}^n \delta_i\tr(h_i h).
\label{eq:kideltai2}
\end{align}
Moreover, the inverse of $\theta_V$ from \Cref{lem:thetaV} is given by 
\begin{equation}
\theta_V^{-1}\colon \CoInd_K^H(V)\to \Ind_K^H(V), \quad f\mapsto \sum_{i=1}^n\delta_i\otimes_K f(h_i).\label{eq:thetainv}
\end{equation}
\end{corollary}

For a Frobenius extension $K\subset H$, we can give $F:=\Ind_K^H$ the structure of a left and right adjoint to $G:=\Res_K^H$. The units and counits of these adjunctions are
\begin{align}
    \unit^{G\dashv F}_W&\colon W\to FG(W), &w&\mapsto \sum_{i=1}^r \delta_i\otimes_K h_iw,\label{eq:unit1}\\ 
    \counit^{G\dashv F}_V&\colon GF(V)\to V, & h\otimes_K v& \mapsto \tr(h)v,\label{eq:counit1}\\
    \unit^{F\dashv G}_V&\colon V\to GF(V), &v&\mapsto 1\otimes_K v,\label{eq:unit2}\\
    \counit^{F\dashv G}_W&\colon FG(W)\to W, &h\otimes_K w&\mapsto hw,\label{eq:counit2}
\end{align}
for any  $K$-module $V$ and any $H$-module $W$.
Moreover, the functor $F$ is lax and oplax monoidal with the lax and oplax monoidal structures
\begin{align}
\begin{split}
    \lax_{V,U}\colon  F(V)\otimes F(U)&\to F(V\otimes U),\\ \quad (h\otimes_K v)\otimes (g\otimes_K u)&\mapsto \sum_i \delta_i\otimes_K (\tr((h_i)_{(1)}h)v\otimes \tr((h_i)_{(2)}g)u),
\end{split}\label{eq:Hopf-lax}\\
\begin{split}
    \oplax_{V,U}\colon  F(V\otimes U)&\to F(V)\otimes F(U),\\ \quad h\otimes_K (v\otimes u)&\mapsto (h_{(1)}\otimes_K v)\otimes (h_{(2)}\otimes_K u),
\end{split}\label{eq:Hopf-oplax}\\
   \lax_0\colon \one &\to F(\one), \quad 1\mapsto \sum_{i=1}^r \delta_i\otimes_K \varepsilon_H(h_i), \label{eq:unit-Hopf}\\
    \oplax_0\colon F(\one)&\to \one, \quad h\otimes_K1 \mapsto \varepsilon_H(h). \label{eq:counit-Hopf}
\end{align}

The following natural isomorphisms, called \emph{projection formula morphisms}, will be important later. The general categorical aspects of the theory behind these morphisms are described in general in \cite{FLP2}, see also \cites{FHM,BLV,FLP3}. 

\begin{definition}\label{def:proj-Hopf}
Given a Frobenius extension $K\subset H$, we define the \emph{left projection formula morphisms} to be the following natural transformations 
\begin{align}
\begin{split}
        \projl{W}{V}^{F, G\dashv F}\colon W\otimes F(V)&\to F(G(W)\otimes V), \\ w\otimes (h\otimes_K v)&\mapsto \sum_{i=1}^r \delta_i\otimes_K ((h_i)_{(1)}w\otimes \tr((h_i)_{(2)}h)v),
\end{split}\\
\begin{split}
        \projl{W}{V}^{F, F\dashv G}\colon F(G(W)\otimes V)&\to W\otimes F(V), \\h\otimes_K (w\otimes v)&\mapsto h_{(1)}w\otimes (h_{(2)}\otimes_K v),
\end{split}
\end{align}
and the \emph{right projection formula morphisms} to be
\begin{align}
\begin{split}
\projr{V}{W}^{F, G\dashv F}\colon F(V)\otimes W&\to F(V\otimes G(W)), \\ (h\otimes_K v)\otimes w&\mapsto \sum_{i=1}^r \delta_i\otimes_K (\tr((h_i)_{(1)}h)v\otimes (h_i)_{(2)}w),
\end{split}
\end{align}
\begin{align}
\begin{split}
    \projr{V}{W}^{F, F\dashv G}\colon F(V\otimes G(W))&\to F(V)\otimes W, \\h\otimes_K (v \otimes w) &\mapsto (h_{(1)}\otimes_K v)\otimes h_{(2)}w.
\end{split}    \label{eq:rproj}
\end{align}
\end{definition}
We note that the pair $\projlnoarg^{F, G\dashv F}$ and $\projrnoarg^{F, G\dashv F}$ comes from the adjunction $G\dashv F$, where $F$ is right adjoint to $G$ while the other pair comes from the adjunction $F\dashv G$, where $F$ is the left adjoint. This is indicated in the superscripts. 

\begin{lemma}\label{lem:proj-Hopf}
The projection formula morphisms are isomorphisms for any Frobenius extension of Hopf algebras. 
\end{lemma}
\begin{proof}
This follows from the results of \cite{FLP2}*{Section~7}.
The invertibility of the projection formula morphisms can also be checked directly using  the inverse antipode. For instance, one checks that 
\begin{align}
\begin{split}
        (\iprojl{W}{V})^{-1}\colon W\otimes F(V)&\to F(G(W)\otimes V), \\w\otimes (h\otimes_K v)&\mapsto h_{(2)}\otimes_K( S^{-1}(h_{(1)})w\otimes  v), 
\end{split}\label{eq:iprojl-inv}\\
\begin{split}
        (\iprojr{V}{W})^{-1}\colon F(V)\otimes W&\to F(V\otimes G(W)), \\(h\otimes_K v)\otimes w &\mapsto h_{(1)}\otimes_K( v\otimes S(h_{(2)})w).
\end{split}\label{eq:iprojr-inv}
\end{align}
Formulas for the inverses for the other two morphisms can be obtained using \Cref{lem:kideltai} but we will not require them. 
\end{proof}

\subsection{The augmented algebra associated with an extension of Hopf algebras}
\label{sec:ovH*}

Let $K\subset H$ be an extension of Hopf algebras. In this section, we introduce the associated augmented algebra $\ov{H}^*$ following \cites{Sch,FMS}.
To do this, we first recall the definition of integrals of augmented algebras.

\begin{definition}[Integrals, unimodularity]
Let $A$ be a $\Bbbk$-algebra with an algebra map $\varepsilon\colon A\to \Bbbk$, called the \emph{augmentation}.
\begin{enumerate}[(i)]
    \item A \emph{right integral} is an element $\Lambda\in A$ such that $\Lambda x=\Lambda \varepsilon(x)$ for all $x\in A$, i.e.\,an element spanning a submodule of $A$ as a right  module over itself, on which $A$ acts by its counit. 
    \item Similarly, a \emph{left integral} is an element $\Lambda\in A$ such that $x\Lambda=\varepsilon(x)\Lambda$ for all $x\in A$. 
    \item An augmented algebra is called \emph{unimodular} if the spaces of left and right integrals coincide.
\end{enumerate}
 \end{definition}

It follows from the definition that the space of right integrals forms a left $A$-submodule of $A$. In particular, if the space of right integrals is one-dimensional, this gives rise to an algebra map $\alpha_A\colon A\to \Bbbk$ determined by the left multiplication
\begin{equation}\label{eq:modularfun}
    a \Lambda = \alpha_A(a)\Lambda, \qquad \forall a\in A,
    \end{equation}
for any non-zero right integral $\Lambda$ of $A$. Note that $\alpha_A$ is independent of the choice of $\Lambda$. The following lemma shows that it detects unimodularity.

\begin{lemma}
    Assume that the spaces of left and right integrals of $A$ are one-dimensional. Then the augmented algebra $A$ is unimodular if and only if $\alpha_A=\varepsilon$.
\end{lemma}
\begin{proof}
    If $A$ is unimodular, then any right integral is also a left integral and hence $\alpha_A=\varepsilon$. 
Conversely, if $\alpha_A=\varepsilon$ and $\Lambda$ is a right integral, then $a\Lambda=\varepsilon(a)\Lambda$ and hence $\Lambda$ is also a left integral. By one-dimensionality of the spaces of integrals, it follows that any left integral equals $s\Lambda$ for a non-zero right integral $\Lambda$ and $s\in \Bbbk$, and hence is a right integral. 
\end{proof}

\begin{remark}\label{rem:Hopf-case}
The above concepts generalize the corresponding tools for Hopf algebras. If $H$ is a finite-dimensional Hopf algebra, then $H$ is an augmented algebra with a one-dimensional space of right integrals \cite{LS2}. Thus, $\alpha_H\in H^*$ is a morphism of algebras, i.e.~$\Delta_{H^*}(\alpha_H)=\alpha_H\otimes \alpha_H$, see e.g.~\cite{Mon}*{Section 2.2}. $\alpha_H$ is called the \emph{distinguished grouplike element} of $H^*$. It follows that $H$ is unimodular if and only if $\alpha_H=\varepsilon_H$.
\end{remark}

\begin{definition}[$\ov{H}$, $\ov{H}^*$] \label{def:overlines}
Let $K\subset H$ be an extension of Hopf algebras and set  $K^+:=\ker\varepsilon_K$.
\begin{enumerate}[(i)]

\item Define $\ov{H}:=H/HK^+$, where $HK^+$ is the left $H$-module generated by $K^+$. Then $\ov{H}$ is a left $H$-module coalgebra with a $K$-linear co-augmentation 
$\nu \colon \Bbbk\to \ov{H}$, $1\mapsto \ov{1}=1_H+HK^+$.

\item The $H$-module quotient map $H\to \ov{H}$, $h\mapsto \ov{h}=h+HK^+$ is a morphism of co-augmented coalgebras, i.e.\,a morphism of coalgebras such that $1_H\mapsto \ov{1}$. 

\item  The $\Bbbk$-dual $\ov{H}^*=\Hom_\Bbbk(H/HK^+,\Bbbk)$ is an augmented algebra with the convolution product
$$
(\mu*\nu)(\ov h) = \mu(\ov{h_{(1)}})\nu(\ov{h_{(2)}}) ,
$$
and unit given by $\varepsilon_{\ov{H}}(\ov{h})=\varepsilon_H(h).$
The augmentation $\varepsilon \colon  \ov{H}^*\to \Bbbk$ is given by evaluation of a linear function $\mu\colon \ov{H}\to \Bbbk$ at $\ov{1}$, i.e.~by $\varepsilon(\mu)=\mu(\ov{1})$.
\end{enumerate}
\end{definition}

\begin{lemma}\label{lem:basisovH}
    Let $K\subset H$ be an extension of Hopf algebras. 
    \begin{enumerate}[(i)]
        \item   For all $h\in H, k\in K$, $\ov{hk}=\varepsilon(k)\ov{h}$. 
        \item   If $H$ is finite over $K$, then $\ov{H}$ is finite-dimensional.
    \end{enumerate}
\end{lemma}
\begin{proof}
To prove Part (i), note that $K^+\subset K$ is a complement to the space $\Bbbk 1_K$ as a $\Bbbk$-vector space. Thus, any $k\in K$ decomposes as 
$k=k'+\varepsilon(k)1$ with $k'\in K^+$ and hence 
$$\ov{hk}=\ov{hk'}+\varepsilon(k)\ov{h}=\varepsilon(k)\ov{h}.$$

To prove Part (ii), assume $H$ is generated by a finite set of elements $\{h_i\}_i$ over $K$. Then for every $h\in H$, there are elements $\{k_i\}_i$ in $K$ such that $h=\sum_i h_i k_i$, and 
$$ \ov{h} = \sum_i \ov{h_i k_i} = \sum_i \varepsilon(k_i) \ov{h_i},$$
by Part (i). 
Hence, $H/HK^+$ is finite-dimensional.
\end{proof}

\subsection{Frobenius extensions of right integral type}
\label{sec:FrobExt}

In this section, we recall extensions of Hopf algebras of right integral type \cite{FMS}*{1.10.\,Definition} and their main properties.

\begin{definition}[Right integral type]\label{def:integral-type}
    We say that an extension $K\subset H$ of Hopf algebras is of \emph{right integral type} if there exists:
    \begin{enumerate}[(a)]
        \item A non-zero right integral $\lambda$ for $\ov{H}^*$;
        \item An element $\Lambda \in H$ such that, for all $h\in H$,
        $$\lambda(\ov{\Lambda h})=\varepsilon(h);$$
        \item A \emph{relative modular function}, 
i.e.\,a morphism of algebras $\chi\colon K\to \Bbbk$ such that, for all $k\in K$, $h\in H$,
\begin{equation}\lambda(\ov{kh})=\chi(k)\lambda(\ov{h}).\label{eq:chi-condition}\end{equation}
    \end{enumerate}
\end{definition}

\begin{remark}
    Note that the conditions (a)--(b) in \Cref{def:integral-type} are existence conditions for certain integral elements related to the Hopf algebra extension $K\subset H$. 
    In particular, if $H$ is finite-dimensional, there exists a non-zero right integral $\Lambda$ for $H$, satisfying $\Lambda h=\Lambda \varepsilon(h)$, and a right integral $\lambda$ for $\ov{H}^*$ such that condition (b) holds. Further, if the space of right integrals for $\ov{H}^*$ is one-dimensional (e.g.\,when $\ov{H}^*$ is finite-dimensional and $HK^+$ is a two-sided ideal and hence $\ov{H}^*$ is a finite-dimensional Hopf algebra), then the relative modular function $\chi$ always exists. In general, the existence of $\lambda,\Lambda,$ and $\chi$ is not guaranteed.
\end{remark}

\begin{remark} \label{rem:uniqueness-chi-1}
    Note also that if $\chi$ as in (c) exists, then it is unique for a given $\lambda$, and in fact, determined by \Cref{eq:chi-condition} for any choice of $\ov h$ such that $\lambda(\ov h)\neq0$. Moreover, if the space of right-integrals of $\ov{H}^*$ is one-dimensional, then $\chi$ is unique. The map $\chi$ is not to be confused with the map $\alpha_K$ from \Cref{eq:modularfun}.
\end{remark}

We recall some results from \cite{FMS}.

\begin{definition}[\cite{FMS}*{1.1.~Definition}] Let $K\subset H$ be an extension of Hopf algebras, let $\beta$ be an algebra automorphism of $K$. Let $K^\beta$ be the right $K$-module whose action is twisted by $\beta$. We say $K\subset H$ is a \emph{$\beta$-Frobenius extension} if $H$ is finite projective over $K$ and $H\cong\Hom_K(H,K^\beta)$ as a $H$-$K$-bimodule. 
\end{definition}

The following argument is taken from \cite{FMS}*{1.8.~Corollary}. 

\begin{lemma} \label{lem:Frob-v-beta-Frob}
Assume $K\subset H$ is a $\beta$-Frobenius extension and there is a linear map $\chi\colon K\to\Bbbk$ such that $\beta(k)=\chi(k_{(1)}) k_{(2)}$ for all $k\in K$. Then the following are equivalent:
\begin{enumerate}[(a)]
    \item $K\subset H$ is a Frobenius extension,
    \item $\chi=\varepsilon$,
    \item $\beta=\id$.
\end{enumerate}
\end{lemma}

\begin{proof} The immediate implications are (b) $\Rightarrow$ (c) $\Rightarrow$ (a), so it remains to prove (a) $\Rightarrow$ (b).

Assume $K\subset H$ is a $\beta$-Frobenius extension and a Frobenius extension. Then we have an isomorphism $\phi\colon H\to H^{\beta}$ of $H$-$K$-bimodules, where the right $K$-action on $H^{\beta}$ is twisted by $\beta$. Set $u:=\phi(1)$. We compute
$$
u\phi^{-1}(1)  = \phi^{-1}(\phi(1)) = 1, 
$$
in particular, $\varepsilon(u)\neq0$. We also compute for all $k\in K$
$$
k u = \phi(k) = u \beta(k) , 
$$
which implies $\varepsilon \beta=\varepsilon$ on $K$. Finally, for all $k\in K$,
$$
\varepsilon\beta(k) = \chi(k_{(1)}) \varepsilon(k_{(2)}) = \chi(k) .
$$
We have shown $\chi=\varepsilon\beta=\varepsilon$.
\end{proof}

We recall the following key result obtained from \cite{Sch}*{3.3.\,Theorem}.

\begin{theorem}[Schneider]\label{thm:Schneider}
    Assume that $K\subset H$ is an extension of Hopf algebras of right integral type such that $H$ is finite faithfully flat over $K$. Then:
    \begin{enumerate}[(a)]
        \item $H$ is finite projective over $K$.
        \item $K\subset H$ is a Frobenius extension if and only if $\chi=\varepsilon_K$.
        \item If the equivalent conditions of (b) hold, a choice of Frobenius morphism in the sense of \Cref{def:Frob-mor} is given by setting 
    \begin{equation}\label{eq:good-Frob-morphism}
    \tr\colon H\to K, \quad \tr(h):= \lambda(\ov{h_{(1)}}) h_{(2)}\in K,
    \end{equation}
    for any non-zero right integral $\lambda$ of $\ov{H}^*$.
    \end{enumerate}
\end{theorem}
\begin{proof}
We apply the statement in \cite{FMS}*{1.14.\,Theorem} (as a right-hand version of \cite{Sch}*{3.3.\,Theorem}) in the special case of $H=W=A$ and $K=U$. Then $A^{\mathrm{co} H}=\Bbbk 1_H$ and the extension $\Bbbk\subset H$ is Galois. Moreover, it follows from \cite{FMS}*{2.4.\,Lemma} that $H^{\mathrm{co} \ov{H}}=K$. 
Thus, \cite{FMS}*{1.14.\,Theorem} implies that $K\subset H$ is a $\beta$-Frobenius extension such that $\beta(k)=\chi(k_{(1)})k_{(2)}$, with $\chi$ the relative modular function. In particular, $H$ is finite projective over $K$, proving (a), and (b) follows from \Cref{lem:Frob-v-beta-Frob}.

Finally, the formula for $\tr$ in (c) is the formula for the Frobenius morphism denoted by $f$ in \cite{FMS}*{1.14.\,Theorem} specified to the above setup. 
    Note that $\tr$ indeed has image contained in $K$ by the argument in \cite{Sch}*{3.1.\,Theorem}.  
\end{proof}

\begin{remark}
In the context of abelian tensor categories, which are, in particular, rigid, \cite{Shi5}*{Section~4.2} provides a generalization of the relative modular function of \cite{FMS} --- the \emph{relative modular object} $\chi_G$ associated with a strong monoidal functor $G\colon \cC\to \cD$ which preserves projective objects. This object $\chi_G$ controls whether $G$ is a Frobenius functor. In the case when $G=\Res_K^H$ for an extension of finite-dimensional Hopf algebras, $\chi_G$ is the one-dimensional $K$-module given by acting via the relative modular function $\chi$ from \Cref{eq:chi-condition}.
\end{remark}

A condition on $H$ that implies that \emph{any} extension $K\subset H$ is of right integral type is cocommutativity of the \emph{coradical} $H_0$ of $H$ (see, e.g.\,\cite{Mon}). 

\begin{corollary}[{\cite{Sch}*{4.3.\,Corollary}}]\label{cor:H0-cocomm}
    Let $H$ be a Hopf algebra such that $H_0$ is cocommutative. Then any finite extension $K\subset H$ of Hopf algebras is of right integral type such that $H$ is faithfully flat over $K$.
\end{corollary}
\begin{proof}
If the coradical of $H$ is cocommutative and $H$ is finite over $K$, then $H$ is faithfully flat over $K$ by \cite{Tak}. That $K\subset H$ is of right integral type was proved in \cite{FMS}*{4.9.\,Corollary}.
\end{proof}

A Hopf algebra is \emph{pointed} if every simple subcoalgebra is one-dimensional.

\begin{corollary}\label{cor:H-pointed}
Let $H$ be a pointed Hopf algebra. Then any finite extension $K\subset H$ is of right integral type and $H$ is free over $K$.    
\end{corollary}
\begin{proof}
    For any pointed Hopf algebra, $H_0$ is a group algebra and hence cocommutative. Freeness of $H$ over $K$ is a theorem from \cite{Rad3}.
\end{proof}

\subsection{Normal bases and extensions of right integral type}\label{sec:FrobExt2}

In the following, we recall classes of extensions of Hopf algebras which are of right integral type. The next definition follows \cite{FMS}*{4.1.\,Definition}.

\begin{definition}[Right normal basis]\label{def:normal-basis}
    We say that $H$ has a \emph{right normal basis} over $K$ if there exists an isomorphism $H\cong \ov{H}\otimes K$ of right $K$-modules and left $\ov{H}$-comodules. Here, $H$ has the induced left $\ov{H}$-comodule structure.
\end{definition}

In particular, the existence of a right normal basis for $H$ over $K$ implies that $H$ is free as a right $K$-module. In particular, $H$ is faithfully flat over $K$. Moreover, \cite{FMS}*{4.8.\,Theorem} proves the following result. 

\begin{proposition}\label{prop:normal-basis-integral}
   Let $K\subset H$ be a finite extension of Hopf algebras. If $H$ has a right normal basis over $K$, then $K\subset H$ is of right integral type and $H$ is free over $K$. 
\end{proposition}

Combining this result with \cite{Sch}*{2.4.\,Theorem} yields the following result.

\begin{corollary}\label{cor:H-fin-dim-integral}
    Assume $H$ is finite-dimensional and $K\subset H$ is a Hopf subalgebra. Then $K\subset H$ is an extension of right integral type such that $H$ is free over $K$.
\end{corollary}
\begin{proof}
    Note that \cite{Sch}*{2.4.\,Theorem} shows that $H$ has a right normal basis over $K$ and is, in particular, free over $K$. 
\end{proof}

\begin{remark} \label{rem:int-1-d}
In the cases when $H$ has a right normal basis over $K$ or $H_0$ is cocommutative, or when $H$ is finite-dimensional or pointed, we have seen that $K\subset H$ is of right integral type (see \Cref{prop:normal-basis-integral} and Corollaries \ref{cor:H0-cocomm}--\ref{cor:H-pointed} and \ref{cor:H-fin-dim-integral}). Moreover, it was shown in \cite{FMS}*{4.9.\,Corollary} that the space of right integrals of $\ov{H}^*$ is one-dimensional and that $\chi$ is the algebra map by which $K$ acts (from the left) on the space of right integrals of $\ov{H}^*$ via \Cref{eq:chi-condition}. In this situation, $\chi$ is unique for the extension $K\subset H$ by \Cref{rem:uniqueness-chi-1}.
\end{remark}

\subsection{Criteria for Frobenius extensions of Hopf algebras}\label{sec:FrobExt3}

We include some sufficient criteria to obtain Frobenius extensions of Hopf algebras.  
The following proposition describes the case where the Hopf algebras are finite-dimensional.

\begin{proposition}[{\cite{FMS}*{1.7.\,Theorem, 1.8.\,Corollary}}]\label{prop:FMS2}
    Let $K\subset H$ be an extension of finite-dimensional Hopf algebras. 
    \begin{enumerate}[(i)]
        \item There exists a right integral  $\lambda_H$ of $H^*$ and  a left integral $\Lambda_K$ of $K$ such that the map $\lambda_{\ov{H}}(\ov{h}):=\lambda_H(h \Lambda_K)$, for $h\in H$, defines a right integral on $\ov{H}^*$. 
        \item $K\subset H$ is a Frobenius extension if and only if $\alpha_H(k)=\alpha_K(k)$ for all $k\in K$.
    \end{enumerate}
Here, $\alpha_H$ is the modular function of the Hopf algebra $H$, see \Cref{rem:Hopf-case}.
\end{proposition}

\begin{corollary}\label{cor:Frobenius-unimodular-case}
Let $H$ be a unimodular Hopf algebra. Then $K\subset H$ is a Frobenius extension if and only if $K$ is unimodular.
\end{corollary}
\begin{proof}
    This follows directly from \Cref{prop:FMS2}(ii) and the fact that a Hopf algebra $H$ is unimodular if and only if $\alpha_H=\varepsilon_H$. 
\end{proof}

\begin{remark}
    If $K\subset H$ is an inclusion of finite-dimensional Hopf algebras, the relative modular function is given by 
    \begin{equation}\chi=(\alpha_H)|_K * \alpha_K^{-1},\label{eq:chi-fd}\end{equation}
for the convolution inverse $\alpha_K^{-1}=\alpha_K S_K$ of the distinguished grouplike of $K^*$ \cite{FMS}*{1.6.\,Definition}, where $S_K$ is the antipode of $K$.
\end{remark}

\begin{example}\label{ex:H-fd}
    If $K=\Bbbk$ and $H$ is a finite-dimensional Hopf algebra, then \Cref{prop:FMS2} shows that for a non-zero right integral $\lambda_H$ of $H^*$, 
    $\tr:=\lambda_H\colon H\to \Bbbk$ defines a Frobenius morphism and $\Bbbk \subset H$ is a Frobenius extension.
\end{example}

To provide another sufficient criterion for Frobenius extensions, recall the concept of a \emph{normal Hopf subalgebra} \cite{Mon}*{3.4.1 Definition}, i.e.\,a Hopf subalgebra $K\subset H$ such that $K$ is a $H$-$H$-subbimodule with respect to the adjoint action. 
In this case, $HK^+=K^+H$ and hence $HK^+$ is a Hopf ideal and $\ov{H}$ a Hopf algebra. For example, if $K$ is contained in the center $Z(H)$ of $H$, then $K$ is a normal Hopf subalgebra.

\begin{corollary}\label{cor:normalHopf}
    Assume $K$ is a normal Hopf subalgebra of $H$ such that $H$ is finite and faithfully flat over $K$. Then $K\subset H$ is a Frobenius extension of right integral type.
\end{corollary}
\begin{proof}
By \cite{FMS}*{1.11.\,Example}, since $K$ is a normal Hopf subalgebra such that $\ov{H}$ is finite-dimensional by \Cref{lem:basisovH}(ii), it follows that $K\subset H$ is an extension of right integral type. As $H$ is faithfully flat over $K$, by \Cref{thm:Schneider}, we know that $H$ is finite projective over $K$ and that, to show that $K\subset H$ is a Frobenius extension, it remains to check that $\chi=\varepsilon_K$.
Since $\ov{H}=H/HK^+$ is a finite-dimensional Hopf algebra, $\ov{H}^*$ has a unique non-zero right integral $\lambda$ up to scalars. Moreover,  using that $\ov{hk}=\ov{h}\varepsilon_K(k)$, for $k\in K$, we compute
\begin{align*}
\lambda(\ov k \ov h) =\lambda(\ov{h_{(1)}S(h_{(2)}) k h_{(3)}})
=\lambda(\ov{h_{(1)}})\varepsilon_K(S(h_{(2)})k h_{(3)})
=\lambda(\ov{h})\varepsilon_K(k)
=\chi(k) \lambda(\ov h).
\end{align*}
This shows that $\chi=\varepsilon_K$, and $K\subset H$ is a Frobenius extension by \Cref{thm:Schneider}.
\end{proof}

The statement of \Cref{cor:normalHopf} simplifies when $H$ is finite-dimensional. In this case, \emph{any} normal Hopf subalgebra $K\subset H$ gives a Frobenius extension since $H$ is free over $K$.

\section{Frobenius monoidal induction functors}\label{sec:main}

This section contains the main results of the paper on the Frobenius monoidal structure of induction along a Frobenius extension of Hopf algebras (\Cref{thm:Hopf-Frobenius-monoidal}) and when this Frobenius monoidal functor extends to categories of YD modules (\Cref{thm:H*-unimodular}). 

\subsection{Conditions for induction functors to be Frobenius monoidal}
\label{sec:ind-frob-mon}

We will now prove that the induction functor $\Ind_K^H$, for a Frobenius extension $K\subset H$ of Hopf algebras of right integral type is a  Frobenius monoidal functor together with the lax and oplax monoidal structures from Equations~\eqref{eq:Hopf-lax}--\eqref{eq:Hopf-oplax}.

We equip $H$ with the regular right $H$-coaction. Then, $K$ has a right $H$-coaction given by extending the regular $K$-coaction via the inclusion $K\subset H$. We recall the following results from \cite{FLP3}*{Proposition~5.13 and Corollary~5.14}.

\begin{proposition}\label{prop:Hopf-right}
Let $K\subset H$ be a Frobenius extension of Hopf algebras. 
\begin{enumerate}[(1)]
    \item 
The following conditions are equivalent.
\begin{enumerate}[(i)]
    \item The natural transformations $\projrnoarg^{F, F\dashv G}$ and $\projrnoarg^{F, G\dashv F}$ from \Cref{def:proj-Hopf} are mutual inverses.
    \item 
    The Frobenius morphism $\tr\colon H\to K$ is a morphism of right $H$-comodules.
\end{enumerate}
\item 
If the equivalent conditions in (1) hold, then $\Ind_K^H\colon \lMod{K}\to \lMod{H}$ is a Frobenius monoidal functor.
\item 
A Frobenius morphism $\tr$ satisfying condition (1)(ii) is unique up to multiplication by a non-zero scalar. 
\end{enumerate}
\end{proposition}

\begin{theorem}\label{thm:Hopf-Frobenius-monoidal}
    Let $K\subset H$ be a Frobenius extension of Hopf algebras of right integral type (see \Cref{def:integral-type}). Then
    $$\Ind_K^H\colon \lMod{K}\to \lMod{H}$$
    is a Frobenius monoidal functor.

    \noindent    This functor restricts to finite-dimensional modules, $\Ind_K^H\colon \lmod{K}\to \lmod{H}$.
\end{theorem}
\begin{proof}
    Assume that $K\subset H$ is a Frobenius extension. Then $H$ is finite projective over $K$. Now, \cite{Sch}*{1.8.\,Corollary} shows that $H$ is faithfully flat over $K$. Thus, \Cref{thm:Schneider} implies that $K\subset H$ admits a Frobenius morphism of the form 
    $$\tr(h)=\lambda(\ov{h_{(1)}}) h_{(2)}.$$ Such a Frobenius morphism is clearly a morphism of right $H$-comodules by coassociatity of the coproduct of $\Delta(h)=h_{(1)}\otimes h_{(2)}$ of $H$. Thus, the claim follows from \Cref{prop:Hopf-right}. As $H$ is, in particular, finite over $K$, it is clear that induction restricts to the subcategories of finite-dimensional modules. 
\end{proof}

\begin{remark}
    By \Cref{thm:Schneider}, \Cref{thm:Hopf-Frobenius-monoidal} says that for any finite faithfully flat extension $K\subset H$ of right integral type such that $\chi=\varepsilon_K$, $\Ind_K^H$ is a Frobenius monoidal functor.  
\end{remark}

The following corollaries of \Cref{thm:Hopf-Frobenius-monoidal} present specific assumptions that ensure that $\Ind_K^H$ is a Frobenius monoidal functor based on the results of Sections \ref{sec:FrobExt}--\ref{sec:FrobExt2}.

\begin{corollary}\label{cor:Ind-Frob-monoidal}
\label{thm-Ind-Frob-mon}
Let $K\subset H$ be a finite Frobenius extension of Hopf algebras. Then $\Ind_K^H$ is Frobenius monoidal in each of the following cases:
\begin{enumerate}[(i)]
    \item $H$ is finite-dimensional.
\item $H$ is pointed.
\item $H$ has a right normal basis over $K$.
\item The coradical of $H$ is cocommutative.
\end{enumerate}
\end{corollary}
\begin{proof}
    In all of these cases (i)--(iv), $K\subset H$ is a faithfully flat extension of right integral type by \cites{Sch,FMS}, see \Cref{prop:normal-basis-integral} and Corollaries \ref{cor:H0-cocomm}, \ref{cor:H-pointed}, and \ref{cor:H-fin-dim-integral}. By assumption, $H\subset K$ is a Frobenius extension. Thus, the result follows from \Cref{thm:Hopf-Frobenius-monoidal}.
\end{proof}

\begin{example}
    If $H$ is finite-dimensional, it follows from \Cref{ex:H-fd} and \Cref{thm-Ind-Frob-mon} that $H^*$ is a Frobenius algebra internal to $\lMod{H}$. Moreover, the isomorphism $\Ind_\Bbbk^H\cong \CoInd_\Bbbk^H$ gives that $H\cong H^*$ as left $H$-modules, which also follows from  $H$ being a Frobenius algebra via the right integral of $H^*$ \cite{LS2}.
\end{example}


\begin{corollary}\label{cor:Ind-Frob-conditions}
Let $K\subset H$ be a finite faithfully flat extension of Hopf algebras. The induction functor $\Ind_K^H$ is Frobenius monoidal provided that one of the following holds:
    \begin{enumerate}[(i)]
        \item $K$ is a normal Hopf subalgebra of $H$,
        \item $K$ is contained in the center $Z(H)$ of $H$,
        \item $H$ is commutative.
    \end{enumerate}
\end{corollary}
\begin{proof}
 As $K\subset H$ is a finite faithfully flat extension, \Cref{cor:normalHopf} implies that $K\subset H$ is a Frobenius extension of right integral type in case condition (i) holds. The cases when (ii) or (iii) hold are special cases as each of these conditions implies (i).
Now, the claim follows from \Cref{thm:Hopf-Frobenius-monoidal}.
\end{proof}

\begin{corollary}\label{cor:Ind-Frob-conditions-fd}
Let $K\subset H$ be an extension of finite-dimensional Hopf algebras. Then the induction functor $\Ind_K^H$ is Frobenius monoidal provided that one of (i)--(iii) from \Cref{cor:Ind-Frob-conditions} or one of the following conditions is satisfied:
    \begin{enumerate}
        \item[(iv)] $\alpha_H(k)=\alpha_K(k)$ for all $k\in K$,
        \item[(v)] both $H$ and $K$ are unimodular,
        \item[(vi)] $H$ is semisimple.
    \end{enumerate}
\end{corollary}
\begin{proof}
Since $H$ is finite-dimensional, the extension $K\subset H$ is finite, free, and of right integral type, see \Cref{cor:H-fin-dim-integral}. Thus, \Cref{cor:Ind-Frob-conditions} applies. Moreover, assuming condition (iv) holds, it follows from \Cref{prop:FMS2} that $K\subset H$ is a Frobenius extension. Condition (v) is a special case of condition (iv). Assuming condition (vi), $K\subset H$ is a Frobenius extension because $H$ being semisimple implies that $K$ is semisimple, see \cite{Mon}*{Section 3.2} and thus both $H$ and $K$ are unimodular. Now, the claim follows from \Cref{thm:Hopf-Frobenius-monoidal}.
\end{proof}

\subsection{Frobenius monoidal functors of Yetter--Drinfeld modules}
\label{sec:FrobMonExt}

In this section, we will introduce a class of  Frobenius extensions of Hopf algebras  $K\subset H$, called \emph{central} Frobenius extensions, for which the functor $\Ind_K^H$ extends to a braided Frobenius monoidal functor on categories of YD modules denoted by $\cZ(\Ind_K^H)$.

\begin{definition}[Central  Frobenius extension]\label{def:central-Frob-ext}
We say that a Frobenius extension $K\subset H$ of Hopf algebras is \emph{central Frobenius extension} if there exists a \emph{central Frobenius morphism} $\tr\colon H\to K$, i.e.\,a Frobenius morphism that is a morphism of $H$-$H$-bicomodules.
\end{definition}

The condition that $\tr\colon H\to K$ is a morphism of $H$-$H$-bicomodules amounts to 
\begin{align}\label{eq:tr-bicomodule-map}
h_{(1)}\otimes \tr(h_{(2)})\otimes h_{(3)}=\tr(h)_{(1)}\otimes \tr(h)_{(2)}\otimes \tr(h)_{(3)}\quad \in \quad H\otimes K\otimes H,
\end{align}
for all $h\in H$, where we identify $K\otimes K\otimes K$ as a subspace of $H\otimes K\otimes H$. Note that if a Frobenius morphism $\tr$ with this property exists, then it is unique up to a non-zero scalar by \Cref{prop:Hopf-right}(3).

\smallskip

The concept of a central Frobenius extension is justified by the following result.

\begin{theorem}[{\cite{FLP3}*{Proposition~5.19 and Corollary~5.20}}]\label{thm:ZInd-Frob-mod}
A Frobenius extension $K\subset H$ is a central Frobenius extension if and only if  the natural transformations $\projlnoarg^{F, F\dashv G}$ and $\projlnoarg^{F, G\dashv F}$ as well as $\projrnoarg^{F, F\dashv G}$ and $\projrnoarg^{F, G\dashv F}$ from \Cref{def:proj-Hopf} are mutual inverses. In this case,
\begin{gather}
\cZ(\Ind_K^H)\colon \lYD{K}\to \lYD{H},\quad (V,\delta)\mapsto (\Ind_K^H(V),\delta^{\Ind_K^H(V)}),\qquad \text{where}\nonumber \\\label{eq:Ind-coaction}
    \delta^{\Ind_K^H(V)}(h\otimes v)=h_{(1)}v^{(-1)}S(h_{(3)})\otimes h_{(2)}\otimes v^{(0)},
\end{gather}
for $v\in V, h\in H$, is a braided Frobenius monoidal functor. Its lax and oplax monoidal structure are given in Equations~\eqref{eq:Hopf-lax}--\eqref{eq:Hopf-oplax}, i.e.\,for YD modules $(V,\delta^V)$, $(W,\delta^W)$ over $K$,
\begin{align}
    \lax_{(V,\delta^V),(W,\delta^W)}=\lax_{V,W}, \qquad     \oplax_{(V,\delta^V),(W,\delta^W)}=\oplax_{V,W}. 
\end{align}
\end{theorem}
\begin{remark}
Note that the lax and oplax monoidal structures do not depend on the coactions of the YD modules $(V,\delta^V),(W,\delta^W)$. In other words, the morphisms $\lax_{V,W}$ and $\oplax_{V,W}$ of $H$-modules are morphisms in $\lYD{H}$ without adjustment. The formula for the coaction in \Cref{eq:Ind-coaction} is obtained from the formulas for the projection formula isomorphisms and their inverses in Equations \eqref{eq:rproj} and \eqref{eq:iprojl-inv}. Again, the lax and oplax monoidal structures depend on a choice of isomorphism $\Ind_K^H\cong \CoInd_K^H$ via a choice of the Frobenius morphism $\tr$, cf.\,\cite{FLP3}*{Remark~5.15}.
\end{remark}

\begin{remark}
Assume that $K\subset H$ is a Frobenius extension of Hopf algebras. Equivalently, we have a natural isomorphism $\Ind_K^H\cong \CoInd_K^H$ by \Cref{thm:ind-coind}. By
    \cite{FLP2}*{Corollary~7.2} and \cite{FLP2}*{Corollary~7.11}, we obtain two different ways to extend $\Ind_K^H$ to a functor on Drinfeld centers, i.e.\,categories of YD modules. One of the functors, obtained from the adjunction $\Ind_K^H\dashv \Res_K^H$, is oplax monoidal and the other functor, obtained  from the adjunction $\Res_K^H\dashv \CoInd_K^H$ via the isomorphism $\Ind_K^H\cong \CoInd_K^H$, is lax monoidal. The two functors only differ in how they extend the coaction of a YD module from $K$ to $H$. The stronger condition that $K\subset H$ is a \emph{central} Frobenius extension ensures precisely that these two ways to extend the coaction are equal. Hence, $\Ind_K^H$ has both a braided lax and oplax monoidal structure and we showed in \cite{FLP3} that these satisfy the compatibilities of a Frobenius monoidal functor, see Equations \eqref{frobmon1}--\eqref{frobmon2}.
\end{remark}

\begin{example}\label{ex:groups2}
    Continuing \Cref{ex:groups}, we see that the Frobenius morphism $\tr\colon H=\Bbbk\sfG\to K=\Bbbk\sfK$ is evidently a morphism of left and right $H$-comodules. Hence, $K\subset H$ is a central Frobenius extension and the induced functor on Drinfeld centers is hence braided Frobenius monoidal, recovering the result from \cite{FHL}*{Proposition~B.1}.
\end{example}

More generally, we can prove that Frobenius extensions of cocommutative Hopf algebras or quasitriangular Hopf algebras, where the inclusion preserves the R-matrix, are central Frobenius extensions. Recall that $1\otimes 1$ is an R-matrix for a cocommutative Hopf algebra.

\begin{corollary}\label{cor:qtriangular-case}
Assume that $K\subset H$ is a Frobenius extension of finite-dimensional Hopf algebras such that $H$ is quasitriangular with an R-matrix $R\in K\otimes K$. Then $K\subset H$ is a central Frobenius extension.
\end{corollary}
\begin{proof}
By assumption, $K\subset H$ is a Frobenius extension. By \Cref{cor:H-fin-dim-integral}, $K\subset H$ is of right integral type and hence, by \Cref{thm:Schneider}, there exists a Frobenius morphism $\tr\colon H\to K$ that is a morphism of right $H$-comodules. The assumption that $R\in K\otimes K$ means that $K$ is also quasitriangular with the same R-matrix and that $G=\Res_K^H$ is a braided strong monoidal functor. Now, by \cite{FLP3}*{Corollary~4.17}, the left projection formula morphisms from \Cref{def:proj-Hopf} are also mutually inverse. This implies, by  \Cref{thm:ZInd-Frob-mod}, that $\tr$ is a morphism of left $H$-comodules and hence of $H$-$H$-bicomodules. Thus, $K\subset H$ is a central Frobenius extension. 
\end{proof}

We conclude this subsection with a detailed example of a Frobenius extension of infinite-dimensional Hopf algebras where $\tr$ is only a morphism of right $H$-comodules but not of left $H$-comodules so that $\Ind_K^H$ is Frobenius monoidal but \Cref{thm:ZInd-Frob-mod} does not apply. A finite-dimensional example was given in \cite{FLP3}*{Example~5.25}.

\begin{example}\label{ex:Taft-infinite}
    Let $\epsilon$ be a primitive $\ell$-th root of unity, with $\ell>1$, and consider the Hopf algebra 
    $$B=\Bbbk \langle g,x\rangle/(gx-\epsilon xg),$$
    which can be seen as an infinite-dimensional analogue of the Taft algebra (which, additionally, quotients out the relations $x^\ell=0$ and $g^\ell=1$). The coproduct, counit, and antipode are defined by the following formulas on generators:
$$\Delta(g)=g\otimes g, \quad \Delta(x)=x\otimes 1+g\otimes x, \quad \varepsilon(g)=1, \quad \varepsilon(x)=0, \quad S(g)=g^{-1}, \quad S(x)=-g^{-1}x.$$
    Inductively, the coproduct is given by 
\begin{align}\label{eq:B-coproduct}
    \Delta(g^ix^j)=\sum_{a=0}^j \binom{j}{a}_\epsilon \epsilon^{i(a-j)}g^{i+j-a}x^a\otimes g^ix^{j-a},
\end{align}
where the $\epsilon$-binomial coefficients are
$$\binom{j}{a}_\epsilon=\frac{[j]_\epsilon!}{[a]_\epsilon![j-a]_\epsilon!},$$
for $[n]_\epsilon=1+\epsilon+\ldots +\epsilon^{n-1}$ and $[n]_\epsilon!=[n]_\epsilon [n-1]_\epsilon \dots [1]_\epsilon$. 
    One checks that the subalgebra $K$ generated by $g^{\pm \ell}$ and $x^\ell$ is central.
    Thus, $K \subset B$ is a Frobenius extension by \Cref{cor:normalHopf}.

A basis for $\ov{B}$ is given by $\Set{\overline{g^ix^j}\mid 0\leq i,j\leq \ell-1}$ and one checks that a right integral is
$$\lambda\colon \ov{B}\to \Bbbk, \quad \overline{g^ix^j}\mapsto \delta_{i,0}\delta_{j,\ell-1},$$
where $\delta_{i,j}$ denotes the Kronecker delta.
\Cref{eq:good-Frob-morphism} yields the Frobenius morphism
    $$\tr\colon B\to K, \quad g^ix^j\mapsto \delta_{i,0}\delta_{j,\ell-1},$$
    for all $0\leq i,j\leq \ell-1$. We compute, using the formula for the coproduct
\eqref{eq:B-coproduct}, that
\begin{align*}
    (g^ix^j)_{(1)}\otimes \tr((g^ix^j)_{(2)})
    &=g^{\ell-1}\otimes 1 \neq 1\otimes 1=\tr(g^ix^j)_{(1)}\otimes \tr(g^ix^j)_{(2)}.
\end{align*}
Hence, $\tr$ is \emph{not} a morphism of left $H$-comodules. As $\tr$ is a morphism of right $H$-comodules, and hence unique up to a scalar by \Cref{prop:Hopf-right}(3), there cannot be a Frobenius morphism that is a morphism of $H$-$H$-bicomodules. Hence,  $K\subset B$ is not a central Frobenius extension. 
\end{example}

\subsection{Unimodularity and semisimplicity conditions for central Frobenius extensions}
\label{sec:cent-Frob-ext}

We now prove the second main theorem of the paper that a Frobenius extension $K\subset H$ of finite-dimensional Hopf algebras is a central Frobenius extension if $K^*$ and $H^*$ are unimodular (\Cref{thm:H*-unimodular}). As a special case, when $H$ is semisimple and co-semisimple, any Hopf subalgebra $K\subset H$ gives a central Frobenius extension (\Cref{cor:H-semisimple}). 
\smallskip 

We first characterize central Frobenius extensions of Hopf algebras of right integral type.

\begin{proposition}\label{prop:lambda-left-right}
Let $K\subset H$ be a finite extension of Hopf algebras of right integral type such that $H$ is faithfully flat over $K$. Then the following are equivalent:
\begin{enumerate}[(a)]
    \item $K\subset H$ is a central Frobenius extension,
    \item 
$\chi=\varepsilon_K$, for the relative modular function $\chi$, and 
\begin{align}
 \lambda(\ov{h_{(1)}}) h_{(2)}=h_{(1)} \lambda(\ov{h_{(2)}}),\label{eq:lambda-left-right}
\end{align}
for all $h\in H$, where $\lambda$ is a non-zero right integral  of $\ov{H}^*$.
 \end{enumerate}
\end{proposition}
\begin{proof}
    Assume $K\subset H$ is a central Frobenius extension. Then by \Cref{thm:Schneider}, $\chi=\varepsilon_K$ and a Frobenius morphism is given by 
    $$\tr(h)=\lambda(\ov{h_{(1)}}) h_{(2)} .$$
    It is clearly a morphism of right $H$-comodules and, hence, unique (up to a non-zero scalar) with this property. As $K\subset H$ is a central Frobenius extension, this Frobenius morphism $\tr$ is also a morphism of left $H$-comodules, and we have
\begin{align*}
 h_{(1)}\otimes \lambda(\ov{h_{(2)}}) h_{(3)}
    =h_{(1)}\otimes \tr(h_{(2)})
     =\tr(h)_{(1)}\otimes \tr(h)_{(2)}
     =\tr(h_{(1)})\otimes h_{(2)}
     =\lambda(\ov{h_{(1)}}) h_{(2)}\otimes h_{(3)},
\end{align*}
where the third equality uses that $\tr$ is a morphism of right $H$-comodules.
Applying $\id_H\otimes \varepsilon_H$ to this equality, we obtain \Cref{eq:lambda-left-right} from the counit axioms. 

Conversely, assume $\chi=\varepsilon_K$ and \Cref{eq:lambda-left-right}. Then $K\subset H$ is a Frobenius extension by \Cref{thm:Schneider} with Frobenius morphism $\tr$ as above and \Cref{eq:lambda-left-right} implies that
$$h_{(1)}\otimes \tr(h_{(2)})=\tr(h_{(1)})\otimes h_{(2)}.$$
Hence, $K\subset H$ is a central Frobenius extension.
    \end{proof}


\begin{corollary}\label{cor:cocommutative-case}
Assume $K\subset H$ is a finite Frobenius extension of Hopf algebras and $H$ is cocommutative. Then $K\subset H$ is a central Frobenius extension.
\end{corollary}
\begin{proof}
As the coradical $H_0$ is cocommutative, $K\subset H$ is of right integral type and $H$ is faithfully flat over $K$ by \Cref{cor:H0-cocomm}. The result now follows from \Cref{prop:lambda-left-right}.
\end{proof}

\begin{example}
    If $H=\Bbbk \sfG$ is a group acting on a Lie algebra $\fr{g}$ by Lie algebra automorphisms and $\sfK$ is a finite-index subgroup, we can form the semidirect product $\Bbbk \sfK \ltimes U(\fr{g}$), cf.\,\cite{Mon}*{Section~5.6}. If $\cha \Bbbk=0$, any pointed cocommutative Hopf algebra is of this form. Now,
    $$\Bbbk \sfK\ltimes U(\fr{g})=K\subset H=\Bbbk \sfG\ltimes U(\fr{g})$$
is a finite extension of right integral type by \Cref{cor:H0-cocomm} since $H$ is cocommutative. Here, $\ov{H}^*=\Bbbk (\sfG/\sfK)$ is the algebra of functions on the quotient $\sfG$-set $\sfG/\sfK$. A right integral is given by the indicator function $\delta_\sfK$ of the coset $\sfK$. For $k\in \sfK,g\in \sfG$, and $u\in \fr{g}$,
    $$\delta_\sfK(\ov{kh})=\delta_{\sfK}(\ov{h})=\varepsilon_K(k)\delta_{\sfK}(\ov{h}), \quad \delta_\sfK(\ov{ku})=0=\varepsilon_K(u)\delta_\sfK(\ov{u}).$$
    This shows that $\chi=\varepsilon_K$ and hence $K\subset H$ is a Frobenius extension. As $H$ is cocommutative, $K\subset H$ is a central Frobenius extension by \Cref{cor:cocommutative-case}. 
\end{example}

Next, we derive the following necessary condition for central Frobenius extensions.

\begin{proposition}\label{prop:Hov*unimodular}
    If $K\subset H$ is a central Frobenius extension of right integral type, then $\ov{H^*}$ is unimodular.
\end{proposition}
\begin{proof}
    In general, the dual $A^*$ of a finite-dimensional coaugmented coalgebra $A$ is unimodular if and only if $\lambda(a_{(1)})\otimes a_{(2)}=a_{(1)}\otimes \lambda(a_{(2)})$ for a non-zero integral $\lambda$ of $A^*$. Applying the map $h\mapsto \ov{h}$ to \Cref{eq:lambda-left-right} thus implies unimodularity of $\ov{H}^*$.
\end{proof}

As the following example shows, the converse of \Cref{prop:Hov*unimodular} is false in the sense that there are Frobenius extensions $K\subset H$ such that $\ov{H}^*$ is unimodular that are \emph{not} central.

\begin{example}
Let $u_\epsilon(\fr{sl}_2)$ the small quantum group of $\fr{sl}_2$ associated with a root of unity $\epsilon$ of order $\ell\geq 2$. Consider the Frobenius extension $\Bbbk \sfC_\ell \subset u_\epsilon(\fr{sl}_2)$ from \cite{FLP3}*{Examples~5.16 and 5.25}. The quotient map $u_\epsilon(\fr{sl}_2)\to \ov{u_\epsilon(\fr{sl}_2)}$ identifies $f^je^ik^m$ and $f^je^i$
and, $$\ov{H}\cong \Bbbk\Set{f^je^i\mid 0\leq i,j\leq \ell-1},$$
as a $\Bbbk$-vector space. Using the coproduct formulas from \cite{Kas}*{Proposition~VII.1.3}, 
we compute the coproduct $\un{\Delta}$ of $\ov{H}$ to be 
$$\un{\Delta}(f^je^i)=\sum_{r=0}^i\sum_{s=0}^j\epsilon^{r(i-r)+s(j-s)-2(i-r)(j-s)}\binom{i}{r}_\epsilon\binom{j}{s}_\epsilon f^se^{i-r}\otimes f^{j-s}e^r.$$
This coalgebra is not a bialgebra as $\Bbbk \sfC_\ell$ is not a normal Hopf subalgebra. 
With the right integral $\lambda(f^je^i)=\delta_{i,\ell-1}\delta_{j,\ell-1}$ for $\ov{u_\epsilon(\fr{sl}_2)}^*$, we compute that
$$(\lambda\otimes \id)\un{\Delta}(f^je^i)=\delta_{i,\ell-1}\delta_{j,\ell-1}\epsilon^{0}\binom{\ell-1}{0}_\epsilon\binom{\ell-1}{\ell-1}_\epsilon f^0e^0=(\id\otimes \lambda)\un{\Delta}(f^je^i).$$
Thus, $\ov{u_\epsilon(\fr{sl}_2)}^*$ is unimodular. However, it was shown in \cite{FLP3}*{Example~5.16} that the associated Frobenius morphism $\tr$ is not a morphism of $u_\epsilon(\fr{sl}_2)$-bicomodules. Thus, $\Bbbk \sfC_\ell \subset u_\epsilon(\fr{sl}_2)$ is \emph{not} a central Frobenius extension. 
\end{example}

\begin{example}
    Continuing \Cref{ex:Taft-infinite}, we note that 
    $\un{\Delta}(\ov{g^ix^j})$ is computed using the same formula \eqref{eq:B-coproduct}, for $0\leq i,j\leq \ell-1$, replacing $g^ax^b$ by $\ov{g^ax^b}$ for all $0\leq a,b\leq \ell-1$. Thus, we check that $\ov{B}^*$ is \emph{not} unimodular since 
    $$(\lambda\otimes \id)\Delta(\ov{g^ix^j})=\ov{1}\otimes \ov{1} \neq \ov{g^{\ell-1}}\otimes \ov{1}=(\id\otimes \lambda)\Delta(\ov{g^ix^j}).$$
This provides another way to see that $K\subset B$ is \emph{not} a central Frobenius extension.
\end{example}

\begin{theorem}\label{thm:H*-unimodular}
    Let $K\subset H$ be a Frobenius extension of finite-dimensional Hopf algebras. If $K^*$ and $H^*$ are unimodular, then $K\subset H$ is a central Frobenius extension.  

\noindent In particular, $\cZ(\Ind_K^H)\colon \lYD{K}\to \lYD{H}$ is a braided Frobenius monoidal functor.
\end{theorem}
\begin{proof}
By \Cref{prop:FMS2}, a right integral for $\ov{H}^*$ is given by the formula 
$$\lambda_{\ov{H}}(\ov{h})=\lambda_H(h\Lambda_K),$$
where $\lambda_H$ is a right integral for $H^*$ and $\Lambda_K$ is a left integral for $K$. Since  $K^*$ is unimodular, the distinguished grouplike element $g_K=1$ and we find that
\begin{equation}
    (\Lambda_K)_{(2)}\otimes S^{-1}((\Lambda_K)_{(1)})=(\Lambda_K)_{(1)}\otimes S((\Lambda_K)_{(2)}),\label{eq:coprod-Lambda}
\end{equation}
for all $k\in K$ by \cite{Rad}*{Theorem~10.5.4}. Denoting $\lambda=\lambda_H$, we compute
\begin{align*}
    \lambda_{\ov{H}}(\ov{h_{(1)}}) h_{(2)}&=\lambda(h_{(1)}\Lambda_K) h_{(2)}\\
    &=\lambda((h(\Lambda_K)_{(1)})_{(1)}) (h(\Lambda_K)_{(1)})_{(2)}S((\Lambda_K)_{(2)})\\
    &=\lambda((h(\Lambda_K)_{(2)})_{(1)}) (h(\Lambda_K)_{(2)})_{(2)}S^{-1}((\Lambda_K)_{(1)})\\
    &=(h(\Lambda_K)_{(2)})_{(1)}S^{-1}((\Lambda_K)_{(1)}) \lambda((h(\Lambda_K)_{(2)})_{(2)})\\
    &=h_{(1)}(\Lambda_K)_{(2)}S^{-1}((\Lambda_K)_{(1)}) \lambda(h_{(2)}(\Lambda_K)_{(3)})\\
    &=h_{(1)} \lambda(h_{(2)}\Lambda_K)=h_{(1)} \lambda_{\ov{H}}(\ov{h_{(2)}}),
\end{align*}
for any $h\in H$.
Here, the second equality uses that for any $h,g\in H$,
$$h_{(1)}g\otimes h_{(2)}=(hg_{(1)})_{(1)}\otimes (hg_{(1)})_{(2)}S(g_{(2)}),$$
 followed by applying \Cref{eq:coprod-Lambda} in the third equality.
The fourth equality uses that 
$\lambda(h_{(1)}) h_{(2)}=h_{(1)} \lambda(h_{(2)}),$ 
for all $h\in H$, by unimodularity of $H^*$, followed by the bialgebra axiom in the fifth, and $h_{(2)}S^{-1}(h_{(1)})=\varepsilon_H(h)1_H$ and the counit axiom in the sixth equality. Thus, by \Cref{prop:lambda-left-right}, $K\subset H$ is a central Frobenius extension. The last conclusion follows from \Cref{thm:ZInd-Frob-mod}.
\end{proof}


\begin{corollary}\label{cor:co-Frobenius}
If $K\subset H$ is a Frobenius extension of finite-dimensional Hopf algebras and $H$ is co-semisimple, then $K\subset H$ is a central Frobenius extension.
\end{corollary}
\begin{proof}
If $H$ is co-semisimple, then $H^*$ is semisimple and hence unimodular by \cite{LS2}*{Propositions 3--4}.
    Moreover, $K^*$ is a quotient of the semisimple Hopf algebra $H^*$ and hence also semisimple, cf.\,\cite{Rad}*{Exercise~10.3.1}. Thus, $K^*$ is also unimodular and the result follows from \Cref{thm:H*-unimodular}.
\end{proof}

The following example shows that even if $H^*$ and $K$ are unimodular ($K$ even semisimple), $K\subset H$ might not be a (central) Frobenius extension of Hopf algebras.

\begin{example}\label{ex:restrictedLie}
Let $\mathfrak{g}$ be a restricted Lie algebra over a field $\Bbbk$ of characteristic $p$ with $p$-map $\mathfrak{g}\to \mathfrak{g}, x\mapsto x^{[p]}$. Denote  its (finite-dimensional) restricted universal enveloping algebra by  $u(\mathfrak{g})$, see e.g.\,\cite{Mon}*{Section~2.3}. Now, by cocommutativity, $u(\mathfrak{g})^*$ is unimodular. 

Consider a restricted Lie subalgebra $\mathfrak{k}\subset \mathfrak{g}$ which is abelian and such that $\mathfrak{k}^{[p]}=\mathfrak{k}$. Then $u(\mathfrak{k})\subset u(\mathfrak{g})$ is a commutative Hopf subalgebra. By \cite{Hoch}, $u(\mathfrak{k})$ is semisimple and hence unimodular. However, $u(\mathfrak{k})\subset u(\mathfrak{g})$ might not be a Frobenius extension of Hopf algebras (unless, say, $\mathfrak{k}\subset \mathfrak{g}$ is normal, i.e.\,closed under the adjoint action). A counterexample is given by the two-dimensional restricted Lie algebra 
$$\mathfrak{g}=\Bbbk x\oplus \Bbbk t, \quad [t,x]=x, \quad t^{[p]}=t, \quad x^{[p]}=0 ,\quad\mathfrak{k}=\Bbbk t\subset\mathfrak{g}.$$
A left (but not right) integral for $u(\mathfrak{g})$ is $(t^{p-1}-1)x^{p-1}$ and a left integral for $u(\mathfrak{k})$ is $t^{p-1}-1$. Thus, 
$$\alpha_{u(\mathfrak{k})}(t)=0\neq 1=\alpha_{u(\mathfrak{g})}(t),$$
see \cite{Far}, and we do not have a Frobenius extension by the converse implication in \Cref{prop:FMS2}(ii). However, considering instead the Hopf subalgebra $u(\mathfrak{h})$, with $\mathfrak{h}=\Bbbk x$,
$$\alpha_{u(\mathfrak{h})}(x)=0=\alpha_{u(\mathfrak{g})}(x)$$
which yields a Frobenius extension of Hopf algebras, and by \Cref{thm:H*-unimodular}, a central Frobenius extension. Note that $u(\mathfrak{h})$ is not semisimple.
\end{example}

\begin{corollary}\label{cor:H-semisimple}
Let $H$ be a finite-dimensional Hopf algebra and $K\subset H$ a Hopf subalgebra. 
\begin{enumerate}[(i)]
    \item If $H$ is semisimple and co-semisimple, then $K\subset H$ is a central Frobenius extension.  
    \item If $\Bbbk$ is a field of characteristic zero and $H$ is semisimple, then $K\subset H$ is a central Frobenius extension. 
\end{enumerate}
In either case, $\cZ(\Ind_K^H)\colon \lYD{K}\to \lYD{H}$ is a braided Frobenius monoidal functor. 
\end{corollary}
\begin{proof}
If $\cha \Bbbk =0$, by \cite{LR1}*{Theorem 3.3}, $H$ is semisimple if and only if it is co-semisimple. Thus, Part (ii) follows from Part (i).

To prove Part (i), assume that $H$ is semisimple and co-semisimple. Then $H$ and $H^*$ are unimodular by \cite{LS2}*{Propositions 3--4}.
Now assume that $K$ is a Hopf subalgebra of $H$. Since $H$ is free over $K$ by \cite{NZ}, any left integral $\Lambda_H$ of $H$ is equal to $\Lambda_K\Lambda$, for a left integral $\Lambda_K$ of $K$ and an element $\Lambda\in H$.
Thus, if $H$ is semisimple, then $K$ is also semisimple (cf.\,\cite{Mon}*{Section 3.2}) and hence unimodular. 

Thus, both $H$ and $K$ are unimodular and we have $\alpha_H=\varepsilon_H$ and $\alpha_K=\varepsilon_K$. This implies that $K\subset H$ is a Frobenius extension by \Cref{prop:FMS2}. As $H$ is co-semisimple by assumption, the result follows from \Cref{cor:co-Frobenius}.
\end{proof}

\section{Classes of examples}\label{sec:examples}

We now present classes of extensions of Hopf algebras involving Drinfeld doubles, quantum groups at roots of unity, and the Kac--Paljutkin algebra.  Throughout this section, we will always use Frobenius morphisms of the form in \Cref{eq:good-Frob-morphism} and the lax and oplax monoidal structures of Equations~\eqref{eq:Hopf-lax}--\eqref{eq:Hopf-oplax}.

\subsection{Drinfeld doubles of unimodular Hopf algebras}\label{sec:Drin-double}

Next, we apply our results to the Drinfeld double $\Drin(H)$ of an unimodular Hopf algebra $H$. We use the presentation of $\Drin(H)$ of \cite{Rad}*{Section~13.1}.

\begin{corollary}\label{cor:DrinH}
Let $H$ be a finite-dimensional unimodular Hopf algebra. Then the extension $H\subseteq \Drin(H)$ is a central Frobenius extension.
\end{corollary}
\begin{proof}
Factorizable Hopf algebras are unimodular \cite{Rad}*{Proposition 12.4.2}. 
 Hence, the Drinfeld double of a finite-dimensional Hopf algebra $H$ is unimodular as it is factorizable \cite{Rad}*{Theorem 3.2.1}. Hence, $\alpha_{\Drin(H)}=\varepsilon$. As $H$ is also unimodular by assumption, we have that $\alpha_H=\varepsilon$. Thus, $\chi=(\alpha_{\Drin(H)})|_H *\alpha_H = \varepsilon $ so that $H\subset \Drin(H)$ is a Frobenius extension.

As shown in \cite{FMS}*{1.9.\,Example}, a right integral $\lambda$ for $\ov{\Drin(H)}^*\cong H^{\oop}$ corresponds to a \emph{left} integral $\Lambda_H$ of $H$ and we may write 
$$\lambda_{\ov{\Drin(H)}}(f\otimes h)=\varepsilon(h)f(\Lambda_H).$$
We now compute, for $f\in H^*$ and $h\in H$,
\begin{align*}
    (\lambda_{\ov{\Drin(H)}}\otimes \id_{\Drin(H)})\Delta_{\Drin(H)}(f\otimes h)&=
     \lambda_{\ov{\Drin(H)}}(f_{(2)}\otimes h_{(1)})\otimes (f_{(1)}\otimes h_{(2)})\\
     &=f_{(2)}(\Lambda_H) f_{(1)}\otimes h\\
     &=f_{(1)}(\Lambda_H) f_{(2)}\otimes h\\
     &=(\id_{\Drin(H)}\otimes \lambda_{\ov{\Drin(H)}})\Delta_{\Drin(H)}(f\otimes h),
\end{align*}
where the second to last equality uses that $H\cong H^{**}$ is unimodular.  Thus, $H\subset \Drin(H)$ is a central Frobenius extension by \Cref{prop:lambda-left-right}.
\end{proof}

In particular, the induction functor 
$$\Ind_H^{\Drin(H)}\colon \lMod{H}\to \lMod{\Drin(H)}\simeq \lYD{H},$$
which is left and right adjoint to the forgetful functor $F\colon \lYD{H}\to \lMod{H}$, is a Frobenius monoidal functor which extends to Drinfeld centers as a braided Frobenius monoidal functor.

\begin{remark}
    Note that $\Drin(H)^*$ will only be unimodular if both $H$ and $H^*$ are unimodular \cite{Rad}*{Proposition 13.4.1}. Hence, \Cref{cor:DrinH} shows that the converse of  \Cref{thm:H*-unimodular} does not hold. For example, the Hopf algebra $u(\mathfrak{g})$ from \Cref{ex:restrictedLie} is not unimodular but its dual is. Thus, $u(\mathfrak{g})^*\subset \Drin(u(\mathfrak{g})^*)\cong \Drin(u(\mathfrak{g}))^{\oop}$ is a central Frobenius extension even though $\Drin(u(\mathfrak{g}))^{\oop}$ is \emph{not} unimodular. 
\end{remark}

\subsection{Quantum groups at roots of unity}\label{sec:quantum-groups}

In this subsection, we apply the results of \Cref{sec:ind-frob-mon} to obtain Frobenius monoidal functors to categories of representations of quantum groups at roots of unity.

We let $U_\epsilon(\fr{g})$ denote the De Concini--Kac--Procesi form of the \emph{quantized enveloping algebra}, or, \emph{quantum group}, associated to a fixed Cartan datum of a semisimple Lie algebra $\fr{g}$ and specialized to a root of unity $\epsilon\in \mC$ of \emph{odd} order $\ell$,\footnote{The order $\ell$ is also required to be prime to $3$ if $\mathfrak{g}$ contains a factor of type $G_2$} see \cite{DeCK} or \cite{CP}*{Section 9.1}.
We let $E_i, F_i, K_i$ denote the usual generators of these quantum groups.

We will also consider the small quantum group $u_\epsilon(\fr{g})$, defined as the quotient of $U_\epsilon(\fr{g})$ by the ideal generated by $E_\alpha^\ell, F_\alpha^\ell, K_i^\ell$, for positive roots $\alpha$ of $\fr{g}$. The finite-dimensional Hopf algebra $u_\epsilon(\fr{g})$ is factorizable \cite{Lyu}*{Corollary A.3.3} and hence unimodular, see, e.g.\,\cite{Rad}*{Proposition 12.4.2}.
We refer to \cite{DGP}*{Section~2.1} for explicit formulas of the braiding, ribbon structure, and integrals for $u_\epsilon(\fr{g})$ and its dual. 

It is well-known that $U_\epsilon(\fr{g})$ contains a large central Hopf subalgebra $Z_0$, which can be identified with the coordinate Hopf algebra $\OH$ of an affine algebraic group $\sfH$, giving an extension of Hopf algebras $\iota \colon \OH\hookrightarrow U_\epsilon(\fr{g})$. Being generated by grouplike and skew-primitive elements, $U_\epsilon(\fr{g})$ is a pointed Hopf algebra and $U_\epsilon(\fr{g})$ is free and finite over $\OH$. 
Thus, by \Cref{cor:H-pointed}, the extension $\OH\subset U_\epsilon(\fr{g})$ is of right integral type and since $\OH\subset Z(H)$, $\OH$ is a normal Hopf subalgebra. Thus, we obtain the following result from \Cref{cor:normalHopf}.

\begin{corollary}\label{cor:quantum-ind}
    The induction functor 
    $$\Ind_{\iota}\colon \lMod{\OH}\to \lMod{U_\epsilon(\fr{g})}$$
    is a Frobenius monoidal functor.
\end{corollary}

Because $\OH$ is a central subalgebra, we have that the ideal $U_\epsilon(\fr{g})\OH^+$ is two-sided. This makes $\ov{U_\epsilon(\fr{g})}$ a Hopf algebra which can be identified with $u_\epsilon(\fr{g})$. Thus, as a right integral, we can use the integral $\lambda$
of $u_\epsilon(\fr{g})^*$ from \cite{DGP}*{Section~2.1} for which, by \Cref{cor:Ind-Frob-conditions}(ii), 
$$\tr(h)=\lambda(\pi(h_{(1)}))\otimes h_{(2)}$$
is a Frobenius morphism, where $\pi\colon U_\epsilon(\fr{g})\to u_\epsilon(\fr{g})$ is the quotient morphism.

\begin{remark}\label{rmk:q-groups}
The functor $\Ind_{\OH}^{U_\epsilon(\fr{g})}$ was shown to induce an oplax monoidal functor $\cZ(\Ind_\iota)$ on the Drinfeld centers in \cite{FLP2}*{Corollary 8.8(ii)}. We note that $\OH\subset U_\epsilon(\fr{g})$ is \emph{not} a central Frobenius extension of Hopf algebras. Indeed, $\tr$ is not a morphism of $H$-$H$-bicomodules by \Cref{prop:lambda-left-right}, since
$$\tr(h)=\lambda(\pi(h_{(1)}))\otimes h_{(2)}\neq h_{(1)}\otimes \lambda(\pi(h_{(2)})).$$
To verify this, consider the unique monomial 
$$X=F_1^{\ell-1}\ldots F_r^{\ell-1}K_{2\rho} E_1^{\ell-1}\ldots E_r^{\ell-1},$$
on which the integral is non-zero, the coproduct takes the form
$$\Delta(X) \in X\otimes K_{2\rho} K_{1}^{\ell-1}\ldots K_r^{\ell-1}+K_{2\rho}K_{1}^{1-\ell}\ldots K_r^{1-\ell}\otimes X +\ker(\lambda\pi)\otimes U_\epsilon(\fr{g})+U_\epsilon(\fr{g})\otimes \ker(\lambda\pi).$$
Thus, applying $\lambda(\pi(h_{(1)}))\otimes h_{(2)}$ and $h_{(1)}\otimes \lambda(\pi(h_{(2)}))$ yields
$$K_{2\rho} K_{1}^{\ell-1}\ldots K_r^{\ell-1}\quad \text{and}\quad  K_{2\rho}K_{1}^{1-\ell}\ldots K_r^{1-\ell}, \quad\text{respectively}.$$ Multiplying with $K_{2\rho}^{-1}$ yields distinct elements in the group algebra generated by the $K_i$, unless $\ell=2$, which contradicts our assumption that $\ell$ is odd. 
Thus, we do \emph{not} obtain a Frobenius monoidal functor $\cZ(\Ind_{\OH}^{U_\epsilon(\fr{g})})$ from \Cref{thm:ZInd-Frob-mod}. 
\end{remark}

\begin{example}
By the above, we see that $\Ind_{\OH}^{U_\epsilon(\fr{g})}(\Bbbk)\cong \ov{U_\epsilon(\fr{g})}\cong u_\epsilon(\fr{g})$ is a Frobenius algebra in $\lmod{U_\epsilon(\fr{g})}$. The $U_\epsilon(\fr{g})$-action is the regular one, restricted along $\pi$. The coproduct is that of the small quantum group $u_\epsilon(\fr{g})$ and the product can be obtained by dualizing the coproduct along the trace map $\tr$ from \Cref{rmk:q-groups}.
\end{example}

\subsection{Hopf subalgebras of small quantum groups}\label{sec:u-subHopf}

Recall that the small quantum group, for a root of unity of odd order $\ell>2$, has a triangular decomposition 
$$u_\epsilon(\fr{g})=u_\epsilon(\fr{n}_-)\otimes \mC\sfK\otimes u_\epsilon(\fr{n}_+),$$
as a $\mC$-vector space,
where $\sfK$ is the group generated by the $K_i$, which is isomorphic to $\mZ_\ell^r$, and $u_\epsilon(\fr{n}_-)$ and $u_\epsilon(\fr{n}_+)$ are the subalgebras generated by the $F_i$ and the $E_i$, respectively.

We will now consider triples $(\Sigma, I_+, I_-)$ defining a Hopf subalgebra of $u_\epsilon(\fr{g})$. Here, $\Sigma\subset \sfK$ is a subgroup and $I_\pm$ are subsets of $\{1,\ldots, r\}$ which are identified with subsets of the simple roots of $\fr{g}$. 
We require that $K_i\in \Sigma$ if $i\in I_+\cup I_-$. It is easy to verify that the subalgebra generated by the $E_i$, for $i\in I_+$, $F_j$, for $j\in I_-$, and $g\in \Sigma$ generate a Hopf subalgebra of $u_\epsilon(\fr{g})$ which we denote by $u_\epsilon(\Sigma, I_+, I_-)$. By  \cite{Mue}*{Theorem~6.3}, see also \cite{AG}*{Corollary~1.13}, all possible Hopf subalgebras of $u_\epsilon(\fr{g})$ are obtained this way.

\begin{example}
       Let $\fr{h}\subset \fr{g}$ be an inclusion of semisimple Lie algebras, defined  by a sub-Cartan datum. Then one obtains a triple $(\Sigma,I_+,I_-)$ as above with $I_+=I_-$ defined such that the simple roots of $\fr g$ indexed by $I_+=I_-$ are the simple roots of $\fr{h}$ and $\Sigma$ is the subgroup of $\sfK$ defining the Cartan part of $u_\epsilon(\fr{h})$. This displays $u_\epsilon(\fr{h})$ as a Hopf subalgebra of $u_\epsilon(\fr{g})$.
\end{example}

We will now determine those Hopf subalgebras of $u_\epsilon(\fr{g})$ that lead to Frobenius extensions. 

\begin{proposition}\label{prop:q-subgroup-unimodular} 
If $I$ is a subset of the simple roots of $\fr{g}$ and $\Sigma\subset \sfK$ is a subgroup containing $(K_i)_{i\in I}$, then $u_\epsilon(\Sigma,I,I)$ is unimodular. Conversely, if $\ell>3$ and  $u_\epsilon(\Sigma, I_+, I_-)$ is unimodular, then $I_+=I_-$.
\end{proposition}

\begin{corollary}\label{cor:q-subgroup-Frobenius}
Assume that $I_+,I_-$ are subsets of the simple roots of $\fr{g}$ and $\ell>3$.
Then, the extension $u_\epsilon(\Sigma,I_+,I_-)\subset u_\epsilon(\fr{g})$ is a Frobenius extension if and only if $I_+=I_-$.
\end{corollary}

\begin{proof}[Proof of \Cref{prop:q-subgroup-unimodular}]
Consider any triple $(\Sigma,I_+,I_-)$ as above. Denote $u'=u_\epsilon(\Sigma, I_+, I_-)$ and $u=u_\epsilon(\fr{g})$. For $I=I_+\cap I_-$, we have a sequence of extensions of Hopf algebras
$$u_\epsilon(\Sigma, I,I)\subset u_\epsilon(\Sigma, I_+,I)\subset u'\subset u.$$
We  first claim that  $u_\epsilon(\Sigma, I,I)$ is unimodular. Indeed, restricting $\Sigma$ to the subgroup generated by those $K_i$ with $i\in I$ we obtain a quantum group $u_\epsilon(\fr{h})$, where $\fr{h}$ is obtained by restricting the Cartan datum to the subset $I$. As $\fr{h}$ is again a semisimple Lie algebra, the Hopf algebra $u_\epsilon(\fr{h})$ is unimodular. Further, arguing as in the proof of \Cref{{cor:H-semisimple}}, a non-zero left integral $\Lambda$ for $u_\epsilon(\Sigma, I,I)$ is obtained as a product $\Lambda_2=\Lambda_1\Lambda'$ where $\Lambda_1$ is a left integral for $u_\epsilon(\fr{h})$ and $\Lambda'\in \mC \Sigma$ (as otherwise $\Lambda_2=0$). We now check that for $K_j$, with $j\notin I$,
$$\Lambda_2K_j=\Lambda_1\Lambda'K_j=\Lambda_1K_j\Lambda'=K_j\Lambda_1\Lambda'=K_j\Lambda_2=\Lambda_2.$$
The second equality uses that $\Lambda'\in \mC \Sigma$, which is commutative. The third equality uses the form of $\Lambda_1$, see \cite{Lyu}, in which for any $E_i$, the corresponding $F_i$ appears with the same power in $\Lambda_1$, and that $E_iK_j=\epsilon^{-i\cdot j}K_jE_i$  and $F_iK_j=\epsilon^{i\cdot j}F_iK_j$. The last equality uses the assumption that $\Lambda_2$ is a left integral for $u_\epsilon(\Sigma, I,I)$. This shows that $u_\epsilon(\Sigma, I,I)$ is unimodular. 

Next, we claim that $u_\epsilon(\Sigma, I_+,I)$ is not unimodular unless $I_+=I$. Assume $\Lambda_3=\Lambda_2\Lambda''$ is a non-zero left integral for $u_\epsilon(\Sigma, I_+,I)$, with $\Lambda_2$ as in the previous paragraph. Given $j\in I_+\setminus I$, we use the triangular decomposition of $u$ to see that $\Lambda''\in \mC\Sigma\prod_{k\in I_+\setminus I}E_k^{\ell-1}$. Thus, $\Lambda_3E_j=0=\varepsilon(E_j)\Lambda_3$ for any $j\in I_+\setminus I$. It also follows that
$$\Lambda_3K_j=\epsilon^{\sum_{k\in I_+\setminus I}(1-\ell)j\cdot k}\Lambda_2K_j\Lambda''=\epsilon^{\sum_{k\in I_+\setminus I}j\cdot k}\Lambda_3,$$
using that $\Lambda_2K_j=\Lambda_2$ as $\Lambda_2$ is a left and right integral for $u_\epsilon(\Sigma, I,I)$. Now, if $\epsilon^{\sum_{k\in I_+\setminus I}j\cdot k}=1$, for all $j\in I_+\setminus I$, then the restriction of the Cartan matrix to the subset $I_+\setminus I$ has only zero row sums when viewed as a matrix with entries in $\mZ_{\ell}$. However, \Cref{lem:Cartan-row-sums}(i) below shows that this is impossible. Hence $\Lambda_3$ is not a right integral and $u_\epsilon(\Sigma, I_+,I)$ is not unimodular.

Finally, again assuming $\ell>3$, we claim that $u'=u_\epsilon(\Sigma, I_+,I_-)$ is unimodular if and only if $I_-=I_+=I$. Let $\Lambda_4=\Lambda_3\Lambda'''$ be a non-zero left integral for $u'$ with $\Lambda_3$ a left integral for $u_\epsilon(\Sigma, I_+,I)$ as in the previous paragraph. Take $F_m$ with $m\in I_-\setminus I$. A similar calculation as in the previous paragraph shows that 
$$\Lambda_4K_m=\epsilon^{-\sum_{n\in I_-\setminus I}m\cdot n}\Lambda_3K_m \Lambda'''=\epsilon^{\sum_{k\in I_+\setminus I}m\cdot k-\sum_{n\in I_-\setminus I}m\cdot n}\Lambda_4.$$
Similarly, we have, for $j\in I_+\setminus I$, that
$$\Lambda_4K_j=\epsilon^{\sum_{k\in I_+\setminus I}j\cdot k-\sum_{n\in I_-\setminus I}j\cdot n}\Lambda_4.$$
Note that $(I_+\setminus I)\cap (I_-\setminus I)$ is empty. Hence, $\Lambda_4$ is a right integral if and only if  the square matrix $d=(d_{a,b})$, with entries in $\mZ_\ell$ and rows and columns indexed by $a,b\in (I_+\setminus I)\cup (I_-\setminus I)$,  defined by 
$$d_{a,b}=\begin{cases}
a\cdot b, &\text{if $b$ in $(I_+\setminus I)$},\\
-a\cdot b, &\text{if $b$ in $(I_-\setminus I)$},
\end{cases}
$$
has only zero row sums. This is not possible as shown in \Cref{lem:Cartan-row-sums}(ii) below, since, by assumption $\ell>3$ is odd.
\end{proof}

\begin{proof}[Proof of \Cref{cor:q-subgroup-Frobenius}]
We know that $u$ is unimodular as it is factorizable. Thus, by \Cref{cor:Frobenius-unimodular-case}, $u'\subset u=u_\epsilon(\fr{g})$ is a Frobenius extension if and only if $u'$ is unimodular. Thus the claim follows as by \Cref{prop:q-subgroup-unimodular}, $u'=u_\epsilon(\Sigma,I_+,I_-)$ is unimodular if and only if  $I_+=I_-=I$. 
\end{proof}

\begin{lemma}\label{lem:Cartan-row-sums}
    Let $C$ be a Cartan matrix of finite type and $\ell> 3$ an odd integer.
    \begin{enumerate}[(i)]
\item 
    Not all row sums of $C$ are equal to zero modulo $\ell$.
\item 
    Now assume $\ell>3$ and consider the modified Cartan matrix $C_J$ in which the columns corresponding to a subset $J$ of the simple roots are multiplied by $-1$. Then, not all row sums of $C_J$ are equal to zero modulo $\ell$.
\end{enumerate}
\end{lemma}
\begin{proof}
First, observe that if the statements hold for all indecomposable Cartan matrices then they hold for decomposable Cartan matrices. 

To prove Part~(i), note that the statement is easily verified for Dynkin diagrams of rank $1$ or $2$. Next, every Dynkin diagram for rank $r\geq 3$ has at least two endpoints with just a single other vertex connected to it. Using a suitable ordering of simple roots, this vertex corresponds to a row $(2,-1,0,\dots,0)$ in the Cartan matrix, whose row sum is non-zero modulo any $\ell\neq 1$.

To prove Part~(ii), multiply each column of the roots from a subset $J$ of the simple roots $I$ by $-1$. 
For rank $1$, the statement is clear since $\ell \neq 2$. For rank $2$, we check all three cases separately. The Cartan matrices are given by 
\begin{align*}
    \begin{pmatrix}
        2\eta_1  & -\eta_2\\
        -\eta_1  & 2\eta_2 
    \end{pmatrix}, \qquad 
        \begin{pmatrix}
        2\eta_1  & -\eta_2\\
        -2\eta_1  & 2\eta_2 
    \end{pmatrix},\qquad   
    \begin{pmatrix}
        2\eta_1  & -\eta_2\\
        -3\eta_1  & 2\eta_2 
    \end{pmatrix}, \qquad \eta_1,\eta_2\in \{\pm 1\}.
\end{align*}
For the second two matrices, it is not possible to choose $\eta_1,\eta_2$ such that both row sums are zero modulo any $\ell\neq 0$. For the first matrix, this is only possible when $\eta_1=-\eta_2$ and $\ell=3$. However, we assume that $\ell>3$ so at least one row sum is non-zero modulo $\ell$ in all cases. 

For rank $r\geq 2$, we can again find an endpoint whose row in the Cartan matrix is $(2,-1,0,\dots,0)$. The corresponding row sum is now given by $2\eta_1-\eta_2$, for some $\eta_1,\eta_2\in\{\pm1\}$.
This row sum can only be zero if $\eta_1=-\eta_2$ and $\ell=3$. As $\ell>3$ by assumption, this case is ruled out as well. 
\end{proof}

\begin{example}
Let $\fr{g}$ be a semisimple Lie algebra.
\begin{enumerate}[(i)]
    \item 
For a Lie subalgebra $\fr{h}\subset \fr{g}$ obtained from a subset of the simple roots,  $u_\epsilon(\fr{h}) \subset u_\epsilon(\fr{g})$ is a Frobenius extension. 
\item 
The inclusion of a group algebra $\mC \Sigma\subset u_\epsilon(\fr{g})$ is Frobenius, corresponding to the case when $I_+=I_-$ are empty.
\end{enumerate}
\end{example}

\begin{example}
    Let $\epsilon$ be a primitive \emph{third} root of unity and consider the Hopf subalgebra $u'$ of $u_\epsilon(\fr{sl}_3)$ generated by $K_1,K_2, E_1, F_2$ only. Let $\Lambda_K$ be a non-zero left and right integral for the Cartan group algebra. Then 
    $$\Lambda=\Lambda_KE_1^2F_2^2$$
    is a left and right integral for $u'$. Thus, $u'\subset u_\epsilon(\fr{sl}_3)$ is a Frobenius extension. This case is excluded in \Cref{prop:q-subgroup-unimodular} by assuming that $\ell>3$.
\end{example}


\begin{corollary}\label{cor:uqg-Frobenius}
 Set $u=u_\epsilon(\fr{g})$ and $u'=u_\epsilon(\Sigma, I,I)$ where $I$ is a subset of the simple roots of $\fr{g}$ and $\Sigma\subset \sfK$ is a subgroup containing $K_i$, for all $i\in I$. Then, the functor 
   $\Ind_{u'}^u$
   is Frobenius monoidal. 
\end{corollary}
\begin{proof}
Since both Hopf algebras are unimodular by \Cref{prop:q-subgroup-unimodular}, we have that the condition from \Cref{prop:FMS2} hold and the inclusion of Hopf algebras is a Frobenius extension and $\chi=\varepsilon_{u'}$ by \Cref{thm:Schneider}. Hence, the statement follows from \Cref{thm:Hopf-Frobenius-monoidal}.
\end{proof}

\begin{remark}
    An inclusion $u_\epsilon(\fr{h})\subset u_\epsilon(\fr{g})$ does not preserve universal R-matrices of these quasitriangular Hopf algebras. Hence, \Cref{cor:qtriangular-case} does not apply here. Indeed,  we do not expect that $\Ind_{u'}^n$ extends to a Frobenius monoidal functor on Drinfeld centers.
\end{remark}

\begin{example}\label{ex:small-Frob-big}
    Consider a semisimple Lie algebra $\fr{g}$ of rank $r$. For $I$ the set of simple roots and $i\in I$, consider the Hopf subalgebra $u_i:=u_\epsilon(\sfK,I\setminus\{i\},I\setminus\{i\})$ of $u=u_\epsilon(\fr{g})$ containing all simple roots besides the $i$-th. Using the PBW basis for $u$, we see that $\Ind_{u_i}^u(\one)$ is a Frobenius algebra of dimension $\dim(u)/\dim(u_i)$.\footnote{
    We thank the anonymous referee for pointing out that this dimension is equal to $\ell^{2(|\Phi^+|-|\Phi^+_{\mathrm{Levi}}|)}$, where $\Phi^+_{\mathrm{Levi}}$ denotes the set of positive roots of the Levi subsystem $\Phi_{\mathrm{Levi}}$ of the roots $\Phi$ of $\fr{g}$ that only involve the simple roots indexed by $I\setminus\{i\}$.  
    }
\end{example}

 \subsection{The Kac--Paljutkin algebra}\label{sec:KP-algebras}

In this section, we study an $8$-dimensional Hopf algebra $H_8$ which gives a central Frobenius extension $K=\Bbbk \sfK\subset H_8$, where $\sfK=\sfC_2\times \sfC_2$ is the Klein four-group and $\cha \Bbbk=0$.

\begin{definition}[The Kac--Paljutkin algebra]
  The \emph{Kac--Paljutkin algebra} $H_8$ of \cite{Mas} is defined as 
 $$H_8=\Bbbk \langle x,y,z ~|~ x^2=y^2=1,~z^2=\frac{1}{2}(1+x+y-xy),~xy=yx,~zx=yz,~zy=xz  \rangle,$$
 and becomes a Hopf algebra with coproduct determined by 
 $$\Delta(x)=x\otimes x, \quad \Delta(y)=y\otimes y,\quad \Delta(z)=\frac{1}{2}(1\otimes 1+1\otimes x+y\otimes 1-y\otimes x)(z\otimes z).$$
 Note that $\varepsilon(x)=\varepsilon(y)=1$, $\varepsilon(z)=0$ determine the counit. 
\end{definition}
 
 A basis for $H_8$ is given by $\{ 1, x, y, z,xy, xz, yz, xyz\}$. The Hopf algebra $H_8$ is a semisimple Hopf algebra that is neither commutative nor cocommutative. 
 The group algebra of the Klein four group $K=\Bbbk\sfK =\Bbbk \langle x,y\rangle$, is a Hopf subalgebra. As $H$ is semisimple, the extension $K\subset H$ is a Frobenius extension by \Cref{cor:Ind-Frob-conditions-fd}.

Note that $\ker\varepsilon_K$ has basis $1-x,1-y,1-xy$ and thus $\ov{H_8}$ has basis $1,z$ (where we use elements from $H_8$ to denote their cosets in $\ov{H_8}$), and both elements are grouplike. We verify that the augmented algebra 
$$\ov{H_8}^*\cong\Bbbk[a]/(a^2-1)\cong \Bbbk \sfC_2,$$
a group algebra. Here, $1=\delta_1+\delta_z$ and $a=\delta_1-\delta_z$ in the dual basis $\{\delta_1,\delta_z\}$ to $\{1,z\}$. An integral is given by $1+a$ which, when viewed as a function on $\ov{H_8}$, corresponds to $2\delta_1$. Thus, by \Cref{eq:good-Frob-morphism}, we have the Frobenius morphism determined by
\begin{equation}\tr\colon H_8\to K, \quad 1\mapsto 2, \quad z\mapsto 0.\label{eq:tr-KP}\end{equation}

Our results imply, as $H_8$ is semisimple, that $K\subset H_8$ is a central Frobenius extension, see \Cref{cor:H-semisimple}. Thus, we obtain a Frobenius monoidal functor 
$$
\cZ(\Ind_K^{H_8})\colon \lYD{K}\to \lYD{H_8}.
$$

By Equations~\eqref{eq:Hopf-lax}--\eqref{eq:Hopf-oplax}, the lax and oplax monoidal structures are given by 
\begin{gather*}
    \lax_{V,W}\colon  H_8\otimes_K V\otimes H_8\otimes_K U\to H_8\otimes_K(V\otimes U),\\
(1\otimes_K v)\otimes (1\otimes_K u)\mapsto 
1\otimes_K (v\otimes u),
\qquad 
(z\otimes_K v)\otimes (z\otimes_K u)\mapsto 
z\otimes_K(v\otimes u),
\\
(1\otimes_K v)\otimes (z\otimes_K u)\mapsto 0,\qquad 
(z\otimes_K v)\otimes (1\otimes_K u)\mapsto 0,
\\
\oplax_{V,W}\colon  H_8\otimes_K(V\otimes U)\to H_8\otimes_K(V)\otimes H_8\otimes_K(U),\\ \quad h\otimes_K (u\otimes v)\mapsto (h_{(1)}\otimes_K v)\otimes (h_{(2)}\otimes_K u),
\\
\lax_0\colon \one \to H_8\otimes_K\one, \quad 1\mapsto \frac{1}{2}(1+z)\otimes_K \one, \quad \oplax_0\colon H_8\otimes_K \one \to \one, \quad h\mapsto \varepsilon(h).
\end{gather*}
Moreover, this Frobenius monoidal functor is \emph{separable}, i.e.\,for any $V,U$ in $\lYD{K}$,
$$\lax_{V,U}\oplax_{V,U}=\id_{F(V\otimes U)}.$$

\begin{lemma}\label{lem:Ind-separable}
The functor 
    $$\cZ(\Ind_{K}^{H_8})\colon \lYD{K}\to \lYD{H_8}$$
    is a separable braided Frobenius monoidal functor.
\end{lemma}
\begin{proof}
\Cref{cor:H-semisimple} implies that $\cZ(\Ind_{K}^{H_8})$ is a braided Frobenius monoidal functor. We can check that the separability condition on the generator $1\otimes_K (v\otimes u)$ of $\Ind_K^{H_8}(V\otimes W)$ as a left $H_8$-module where it is clear.
\end{proof}

Assume now that $\Bbbk$ is algebraically closed and recall that $\cha \Bbbk=0$. The simple YD modules over the Kac--Paljutkin algebra $H_8$ were explicitly described in \cites{HY,Shi}. With the notation of \cite{Shi}, there are five non-isomorphic simple $H_8$-modules
$$V_1(b), \quad b\in \Set{\pm1, \pm i}, \quad V_2,$$
where on the one-dimensional simples $V_1(b)$, $z$ acts by $b$ and $x,y$ act by $b^2$. For the $2$-dimensional simple $V_2=\Bbbk\Set{v_1,v_2}$, we have
$$zv_1=v_1,\quad zv_2=-v_2, \quad xv_1=-v_2, \quad yv_1=v_2.$$
The simple $K$-modules are one-dimensional and denoted by $\Bbbk_{\epsilon_x,\epsilon_y}$ where $x$, $y$ act by $\epsilon_x,\epsilon_y\in \Set{\pm 1}$, respectively. We compute for $\Ind=\Ind_K^{H_8}$ that
$$\Ind(\Bbbk_{1,1})\cong V_1(1)\oplus V_1(-1), \quad \Ind(\Bbbk_{1,-1})\cong \Ind(\Bbbk_{-1,1})\cong V_2,\quad \Ind(\Bbbk_{-1,-1})\cong V_1(i)\oplus V_1(-i).$$
Thus, every simple $H_8$-module is a summand of an object in the image of $\Ind$.

We observe that 
    a YD module $(V,\delta_V)$ over $H_8$ appears in the closure of the image of $\cZ(\Ind_K^{H_8})$ under direct summands if and only if its coaction $\delta_V$ factors through $K\otimes V$, for $K=\Bbbk \sfK$.
In the list of simple YD modules over $H_8$ from \cite{Shi}*{Theorem 3.13}, we can identify the images of the simple YD modules over $K$ under $\cZ(\Ind)=\cZ(\Ind_K^{H_8})$ as follows:
\begin{align*}
\cZ(\Ind)(\Bbbk_{1,1}^1)&\cong M\inner{1,1}\oplus M\inner{-1,1}, & \cZ(\Ind)(\Bbbk_{1,1}^x)&\cong M\inner{(x,y)},\\
\cZ(\Ind)(\Bbbk_{1,1}^{xy})&\cong M\inner{1,xy}\oplus M\inner{-1,xy}, & \cZ(\Ind)(\Bbbk_{1,1}^y)&\cong M\inner{(x,y)},\\
\cZ(\Ind)(\Bbbk_{-1,-1}^x)&\cong M\inner{i,x}\oplus M\inner{-i,x}, & \cZ(\Ind)(\Bbbk_{-1,-1}^1)&\cong M\inner{(1,xy)},\\
\cZ(\Ind)(\Bbbk_{-1,-1}^{y})&\cong M\inner{i,y}\oplus M\inner{-i,y}, & \cZ(\Ind)(\Bbbk_{-1,-1}^{xy})&\cong M\inner{(1,xy)},\\
\cZ(\Ind)(\Bbbk_{1,-1}^{1})&\cong M\inner{(1,y))}, & \cZ(\Ind)(\Bbbk_{1,-1}^{x})&\cong M\inner{(x,1)},\\
\cZ(\Ind)(\Bbbk_{1,-1}^{y})&\cong M\inner{(y,xy))}, & \cZ(\Ind)(\Bbbk_{1,-1}^{xy})&\cong M\inner{(xy,x)},\\
\cZ(\Ind)(\Bbbk_{-1,1}^{y})&\cong M\inner{(1,y))}, & \cZ(\Ind)(\Bbbk_{-1,1}^{1})&\cong M\inner{(x,1)},\\
\cZ(\Ind)(\Bbbk_{-1,1}^{xy})&\cong M\inner{(y,xy))}, & \cZ(\Ind)(\Bbbk_{-1,1}^{x})&\cong M\inner{(xy,x)}.
\end{align*}
This list shows that all eight $1$-dimensional simples 
$$M\inner{a,b}, \text{ with } \inner{a,b}=\inner{\pm 1,1},\inner{\pm 1,xy},\inner{\pm i,x},\inner{\pm i,y},$$
and a family of six $2$-dimensional simples 
$$M(\inner{c,d}), \text{ with } \inner{c,d}=\inner{x,y},\inner{1,xy},\inner{x,1},\inner{1,y},\inner{xy,x},\inner{y,xy},$$
appear as summands. There are eight further simple YD modules over $H_8$ with coactions that do not factor over the group algebra $K$ and that do not appear in the image of $\cZ(\Ind)$.

\smallskip

As an application, we can now construct commutative Frobenius algebras in  $\lYD{H_8}$.
\begin{example}\label{ex:Zone-H}
Let $\one\in \lYD{K}$ be the trivial Frobenius algebra. Then $A:=\cZ(\Ind_K^{H_8})(\one)$ is the two-dimensional Frobenius algebra $\ov{H_8}^*=\Bbbk[a]/(a^2-1)$ in $\cZ(\lmod{H_8})$, where $1$, $a$ correspond to $\tfrac{1}{2}(1+z)\otimes_K 1$, $\tfrac{1}{2}(1-z)\otimes_K 1$, respectively.
The coproduct is given by
\begin{align*}
    \Delta_A(a)=&\frac{1}{2}\oplax_{\one,\one}((1-z)\otimes_K 1)
=1_A\otimes a+a\otimes 1_A,
\end{align*}
using that 
$\frac{1}{2}(1_A+a)=1$, and  $ \frac{1}{2}(1_A-a)=z.$ Moreover, 
\begin{align*}
    \Delta_A(1_A)=&\frac{1}{2}\oplax_{\one,\one}((1+z)\otimes_K 1)
    =1_A\otimes 1_A+a\otimes a. 
\end{align*}
The counit for the Frobenius algebra is thus given by 
$\epsilon_A(1_A)=1$ and $\epsilon_A(a)=0.$
The $H_8$-Yetter--Drinfeld module structure on $A$ consists of the trivial $H_8$-coaction and the $H_8$-action
$$z\cdot a=-a, \quad x\cdot a=y\cdot a=a.$$
\end{example}

\subsection{Rigid Frobenius algebras in the Drinfeld center of the Kac--Paljutkin algebra}\label{sec:Frob-alg-KP}

As an application of our results, we conclude by constructing Frobenius algebras in $\cZ(\lmod{H_8})$ whose categories of local modules are modular fusion categories. We now assume that $\cha\Bbbk=0$ and $\Bbbk$ is algebraically closed.

An \emph{\'etale algebra} in an abelian tensor category $\cC$ is a commutative separable algebra \cite{DMNO}*{Definition~3.1}. We say that $A$ is \emph{connected} if $\Hom_{\cC}(\one,A)=\Bbbk$.
Any connected \'etale algebra in a fusion category is a Frobenius algebra, see, e.g.\,\cite{LW3}*{Proposition~3.11}. Connected \'etale algebras in $\cZ(\lmod{\Bbbk \sfG})$, for a finite group $\sfG$, were classified in \cite{Dav}. These are described by a normal subgroup $\sfN\triangleleft \sfH$, for a subgroup $\sfH\subset \sfG$ and cocycle data, see \cite{LW3}*{Section~6.2} for details. 

We now apply this classification to $\sfK=\sfC_2\times \sfC_2$. In case $\sfN=\{1\}$ we only get the trivial algebra $\one$ in $\cZ(\lmod{\Bbbk \sfK})$ and its image in $\cZ(\lmod{H_8})$ was considered in \Cref{ex:Zone-H}. If $\sfN=\sfH$, the categories of local modules are trivial, see \cite{DS} or \cite{LW3}*{Corollary~6.18(3)--(4)}. We will not compute the images of these so-called \emph{Lagrangian subalgebras} under $\cZ(\Ind_{\Bbbk \sfK}^{H_8})$ here.
Hence, we will consider the cases
$$\sfH=\sfK=\sfC_2\times \sfC_2\quad \text{ and }\quad\sfN=\inner{n}\cong \sfC_2, \quad \text{ where $n=x$, $y$, or $xy$}.$$
The triples $\{(n,\epsilon_x,\epsilon_y):n\in\{x,y,xy\},\epsilon_x,\epsilon_y\in\{\pm1\}\}$  parametrize a discrete family of twelve non-isomorphic two-dimensional \'etale algebras $A_{n,\epsilon_x,\epsilon_y}$ in $\cZ(\lmod{\Bbbk \sfK})$. We use the basis $\Set{e_1,e_n}$, with $\deg e_1=1$ and $\deg e_n=n$. 
The action is given by $x\cdot e_m=\epsilon_x e_m$, $y\cdot e_m=\epsilon_x e_m$, for some $\epsilon_x,\epsilon_y\in \Set{\pm 1}$, where $m=1,n$.
These algebras are Frobenius algebras with coproduct and counit 
$$\Delta(e_m)=e_1\otimes e_{m}+e_n\otimes e_{nm},$$
for $m=1,n$, by \cite{LW3}*{Proposition 6.12}.

Our results imply that these algebras induce four-dimensional commutative algebras in $\cZ(\lmod{H_8})$.
We denote these algebras by 
$$B_{n,\epsilon_x,\epsilon_y}=\cZ(\Ind_{\Bbbk \sfK}^{H_8})(A_{n,\epsilon_x,\epsilon_y}).$$
A basis for each of these algebras is given by 
$$\{ b_{1m}:=1\otimes e_m,\quad  b_{zm}:=z\otimes e_m~|~ m=1,n \}.$$

\begin{lemma}\label{lem:B-commFrob}
    The algebra $B_{n,\epsilon_x,\epsilon_y}$ is a commutative Frobenius algebra in $\cZ(\lmod{H_8})$. Explicitly, for $m,m'\in \Set{1,n}$:
    \begin{enumerate}[(i)]
        \item The $H_8$-action is given by the induced action, i.e.
        \begin{gather*}
        z\cdot b_{1m}=b_{zm}, \quad z\cdot b_{zm}=\frac{1+\epsilon_x+\epsilon_y-\epsilon_x\epsilon_x}{2}b_{1m}=b_{1m},\\
        x\cdot b_{11}=b_{11}, \quad x\cdot b_{z1}=b_{z1},\quad 
        y\cdot b_{11}=b_{11}, \quad y\cdot b_{z1}=b_{z1},\\
        x\cdot b_{1n}=\epsilon_{x}b_{1n}, \quad x\cdot b_{zn}=\epsilon_{y}b_{zn},\quad 
        y\cdot b_{1n}=\epsilon_{y}b_{1n}, \quad y\cdot b_{zn}=\epsilon_{x}b_{zn}.
        \end{gather*}
        \item The $H_8$-coaction is determined by 
        \begin{gather*}
            \delta(b_{1m})=m\otimes b_{1m}, \quad \delta(b_{zm})=
            m(1+\epsilon_xx+\epsilon_xy -\epsilon_x\epsilon_xxy)\otimes b_{zm}.
        \end{gather*}
        \item The multiplication is determined by 
        \begin{gather*}
            b_{1m}b_{1m'}=b_{1,mm'}, \quad b_{zm}b_{1m'}=b_{1m}b_{zm'}=0,\quad
            b_{zm}b_{zm'}=b_{z,mm'}.
        \end{gather*}
        Thus, $1=b_{11}+b_{z1}$ is a central idempotent decomposition of the unit and as  a $\Bbbk$-algebra, $B_{n,\epsilon_x,\epsilon_y}$ is isomorphic to $\Bbbk[a]/(a^2-1)\times \Bbbk[b]/(b^2-1)$.
    \end{enumerate}
\end{lemma}


Recall the definition of a \emph{rigid Frobenius algebra} in a ribbon category $\cD$ \cite{LW3}*{Definition~3.8}.\footnote{In a ribbon fusion category, such an algebra is, equivalently, a \emph{rigid $\cD$-algebra} in the sense of \cite{KO}.} A rigid Frobenius algebra $A$ is, equivalently, a connected, commutative, \emph{special} Frobenius algebra, i.e.
$$m_A\Delta_A=\beta_A\id_A, \quad \epsilon_A u_A=\beta_\one\id_\one,$$
for $\beta_A,\beta_\one\in \Bbbk^\times$.
The categories of internal modules $\rModint{\cD}{A}$ over such an algebra $A$ provide a source of tensor categories and their centers are given by the categories $\rModloc{\cD}{A}$ of so-called \emph{local modules} \cites{Sch, KO, LW3}.

The Drinfeld center $\cZ(\lmod{H_8})$ of the category of finite-dimensional $H_8$-modules is a ribbon category.
Indeed, $H_8$ is a spherical Hopf algebra as it is involutory, i.e.\,$S^2=\id_H$ for the antipode $S$ \cite{BW}*{Example~3.2}. Thus, the ribbon structure on the center of $\lmod{H_8}$ is obtained from the spherical structure of $\lmod{H_8}$ via \cite{Shi4}. As $\cha \Bbbk =0$, the category $\lmod{H_8}$ is fusion and so is its Drinfeld center \cite{ENO}*{Theorem~2.15}. 

\begin{corollary}\label{cor:loc-mod-H8}
    The algebras $B=B_{n,\epsilon_x,\epsilon_y}$ are rigid Frobenius algebras in $\cZ(\lmod{H_8})$. In particular, the categories $\rModint{\lmod{H_8}}{B}$ are fusion categories and their Drinfeld centers 
    $$\cZ(\rModint{\lmod{H_8}}{B})\simeq\rModloc{\cZ(\lmod{H_8})}{B}$$ 
    are modular fusion categories. 
\end{corollary}
\begin{proof}
As $A_{n,\epsilon_x,\epsilon_y}=\one\oplus \Bbbk_{\epsilon_x,\epsilon_y}^n$, our explicit calculations after \Cref{lem:Ind-separable} show that $B=\cZ(\Ind_K^{H_8})(A_{n,\epsilon_x,\epsilon_y})$ only contains a single copy of the tensor unit $\one\in \lmod{H_8}$ and is, hence, connected.
    It remains to show that $B$ is special. By definition, this amounts to 
    the composition $m_B\Delta_B$ of coproduct and product being a non-zero multiple of the identity. We know by \cite{LW3}*{Theorem~6.14} that for the algebras $A=A_{n,\epsilon_x,\epsilon_y}$, 
    $$m_A\Delta_A=|\sfN|\id_{A}=2\id_A.$$
    As, by \Cref{lem:Ind-separable}, $\lax_{A,A}\oplax_{A,A}=\id$, it follows that 
\begin{align*}m_B\Delta_B&=\Ind(m_A)\lax_{A,A}\oplax_{A,A}\Ind(\Delta_A)=2\id_{\Ind(A)}=2\id_B,\\
\varepsilon_Bu_B&=\oplax_0 (b_{11}+b_{z1})=\varepsilon_{H_8}(1)+\varepsilon_{H_8}(z)=1,
\end{align*}
showing that $B$ is a special Frobenius algebra. 

The non-degenerate braided tensor categories $\rModloc{\cZ(\lmod{H_8})}{B}$ inherit the ribbon structure of $\cZ(\lmod{H_8})$ by \cite{KO} and are, hence, modular tensor categories. 
\end{proof}

\subsection{Further questions}\label{sec:open-questions}

Here, we collect some open questions and directions not investigated in this paper and \cite{FLP3}.\footnote{Since completion of this article, Jaklitsch and Yadav in \cite{JY}*{Theorem~3.23} have positively answered Questions \ref{question1} and \ref{question2}.}

\begin{question}\label{question1}
Assume that $\cha \Bbbk=0$.
    If $G\colon \cC\to \cD$ is a strong monoidal functor of fusion categories, is $\cZ(R)$ always a Frobenius monoidal functor?
\end{question}

For fusion categories of the form $\cC=\lmod{H}$, for a finite-dimensional semisimple Hopf algebra $H$ over $\Bbbk$, with $\cha \Bbbk=0$, it follows from \Cref{cor:H-semisimple} and \Cref{thm:ZInd-Frob-mod} that  \Cref{question1} is answered in the affirmative for tensor functors induced by extensions of Hopf algebras. The question can be asked, specifically, for categories of modules over semisimple \emph{quasi}-Hopf algebras.

\begin{question}\label{question2}
    If $K\subset H$ is an inclusion of semisimple finite-dimensional quasi-Hopf algebras over a field or characteristic zero, is $\cZ(\Ind_K^H)$ a Frobenius monoidal functor?
\end{question}

Another interesting question may be the explicit computation of the lax and oplax monoidal structure for the functor in \Cref{cor:quantum-ind} and \Cref{cor:uqg-Frobenius} and the resulting Frobenius algebras in $\lMod{U_\epsilon(\fr{g})}$ and $\lMod{u_\epsilon(\fr{g})}$, respectively.

Finally, concerning the commutative Frobenius algebras $B=B_{m,\epsilon_x,\epsilon_y}$ from \Cref{sec:KP-algebras}, it would be interesting to identify their tensor categories of modules and their modular categories of local modules from \Cref{cor:loc-mod-H8} among known modular fusion categories, e.g.\,as Drinfeld centers of finite-dimensional modules over (quasi-)Hopf algebras.

\bibliography{biblio}

\begin{bibdiv}
\begin{biblist}

\bib{AG}{article}{
      author={Andruskiewitsch, Nicol\'{a}s},
      author={Garc\'{\i}a, Gast\'{o}n~Andr\'{e}s},
       title={Quantum subgroups of a simple quantum group at roots of one},
        date={2009},
        ISSN={0010-437X},
     journal={Compos. Math.},
      volume={145},
      number={2},
       pages={476\ndash 500},
         url={https://doi.org/10.1112/S0010437X09003923},
}

\bib{Bal}{article}{
      author={Balan, Adriana},
       title={On {H}opf adjunctions, {H}opf monads and {F}robenius-type
  properties},
        date={2017},
        ISSN={0927-2852},
     journal={Appl. Categ. Structures},
      volume={25},
      number={5},
       pages={747\ndash 774},
         url={https://doi.org/10.1007/s10485-016-9428-0},
}

\bib{BLV}{article}{
      author={Brugui\`eres, Alain},
      author={Lack, Steve},
      author={Virelizier, Alexis},
       title={Hopf monads on monoidal categories},
        date={2011},
        ISSN={0001-8708},
     journal={Adv. Math.},
      volume={227},
      number={2},
       pages={745\ndash 800},
         url={https://doi.org/10.1016/j.aim.2011.02.008},
}

\bib{BW}{article}{
      author={Barrett, John~W.},
      author={Westbury, Bruce~W.},
       title={Spherical categories},
        date={1999},
        ISSN={0001-8708},
     journal={Adv. Math.},
      volume={143},
      number={2},
       pages={357\ndash 375},
         url={https://doi.org/10.1006/aima.1998.1800},
}

\bib{CMZ}{book}{
      author={Caenepeel, Stefaan},
      author={Militaru, Gigel},
      author={Zhu, Shenglin},
       title={Frobenius and separable functors for generalized module
  categories and nonlinear equations},
      series={Lecture Notes in Mathematics},
   publisher={Springer-Verlag, Berlin},
        date={2002},
      volume={1787},
        ISBN={3-540-43782-7},
         url={https://doi.org/10.1007/b83849},
}

\bib{CP}{book}{
      author={Chari, Vyjayanthi},
      author={Pressley, Andrew},
       title={A guide to quantum groups},
   publisher={Cambridge University Press, Cambridge},
        date={1994},
        ISBN={0-521-43305-3},
}

\bib{Dav}{article}{
      author={Davydov, Alexei},
       title={Modular invariants for group-theoretical modular data. {I}},
        date={2010},
        ISSN={0021-8693},
     journal={J. Algebra},
      volume={323},
      number={5},
       pages={1321\ndash 1348},
         url={https://doi.org/10.1016/j.jalgebra.2009.11.041},
}

\bib{DeCK}{incollection}{
      author={De~Concini, Corrado},
      author={Kac, Victor~G.},
       title={Representations of quantum groups at roots of {$1$}},
        date={1990},
   booktitle={Operator algebras, unitary representations, enveloping algebras,
  and invariant theory ({P}aris, 1989)},
      series={Progr. Math.},
      volume={92},
   publisher={Birkh\"{a}user Boston, Boston, MA},
       pages={471\ndash 506},
}

\bib{DMNO}{article}{
      author={Davydov, Alexei},
      author={M\"{u}ger, Michael},
      author={Nikshych, Dmitri},
      author={Ostrik, Victor},
       title={The {W}itt group of non-degenerate braided fusion categories},
        date={2013},
        ISSN={0075-4102},
     journal={J. Reine Angew. Math.},
      volume={677},
       pages={135\ndash 177},
         url={https://doi.org/10.1515/crelle.2012.014},
}

\bib{DP}{article}{
      author={Day, Brian},
      author={Pastro, Craig},
       title={Note on {F}robenius monoidal functors},
        date={2008},
     journal={New York J. Math.},
      volume={14},
       pages={733\ndash 742},
         url={http://nyjm.albany.edu:8000/j/2008/14_733.html},
}

\bib{DGP}{article}{
      author={De~Renzi, Marco},
      author={Geer, Nathan},
      author={Patureau-Mirand, Bertrand},
       title={Nonsemisimple quantum invariants and {TQFT}s from small and
  unrolled quantum groups},
        date={2020},
        ISSN={1472-2747},
     journal={Algebr. Geom. Topol.},
      volume={20},
      number={7},
       pages={3377\ndash 3422},
         url={https://doi.org/10.2140/agt.2020.20.3377},
}

\bib{DS}{article}{
      author={Davydov, Alexei},
      author={Simmons, Darren},
       title={On {L}agrangian algebras in group-theoretical braided fusion
  categories},
        date={2017},
        ISSN={0021-8693},
     journal={J. Algebra},
      volume={471},
       pages={149\ndash 175},
         url={https://doi.org/10.1016/j.jalgebra.2016.09.016},
}

\bib{EGNO}{book}{
      author={Etingof, Pavel},
      author={Gelaki, Shlomo},
      author={Nikshych, Dmitri},
      author={Ostrik, Victor},
       title={Tensor categories},
      series={Mathematical Surveys and Monographs},
   publisher={American Mathematical Society, Providence, RI},
        date={2015},
      volume={205},
        ISBN={978-1-4704-2024-6},
         url={https://doi.org/10.1090/surv/205},
}

\bib{ENO}{article}{
      author={Etingof, Pavel},
      author={Nikshych, Dmitri},
      author={Ostrik, Viktor},
       title={On fusion categories},
        date={2005},
        ISSN={0003-486X},
     journal={Ann. of Math. (2)},
      volume={162},
      number={2},
       pages={581\ndash 642},
         url={https://doi.org/10.4007/annals.2005.162.581},
}

\bib{Far}{misc}{
      author={Farnsteiner, Rolf},
       title={Hopf modules and integrals: The space of integrals},
        date={2006},
        note={Lecture Notes available at
  \url{http://www.mathematik.uni-bielefeld.de/~sek/selected.html}},
}

\bib{Farn}{article}{
      author={Farnsteiner, Rolf},
       title={On {F}robenius extensions defined by {H}opf algebras},
        date={1994},
        ISSN={0021-8693},
     journal={J. Algebra},
      volume={166},
      number={1},
       pages={130\ndash 141},
         url={https://doi.org/10.1006/jabr.1994.1144},
}

\bib{FHL}{article}{
      author={Flake, Johannes},
      author={Harman, Nate},
      author={Laugwitz, Robert},
       title={The indecomposable objects in the center of {D}eligne's category
  {$\underline{\rm Re}{\rm p}S_t$}},
        date={2023},
        ISSN={0024-6115},
     journal={Proc. Lond. Math. Soc. (3)},
      volume={126},
      number={4},
       pages={1134\ndash 1181},
         url={https://doi.org/10.1112/plms.12509},
}

\bib{FHM}{article}{
      author={Fausk, H.},
      author={Hu, P.},
      author={May, J.~P.},
       title={Isomorphisms between left and right adjoints},
        date={2003},
     journal={Theory Appl. Categ.},
      volume={11},
       pages={No. 4, 107\ndash 131},
}

\bib{FLP2}{article}{
      author={Flake, Johannes},
      author={Laugwitz, Robert},
      author={Posur, Sebastian},
       title={Projection formulas and induced functors on centers of monoidal
  categories},
        date={2024},
     journal={arXiv e-prints},
       pages={ArXiv:2402.10094},
}

\bib{FLP3}{article}{
      author={Flake, Johannes},
      author={Laugwitz, Robert},
      author={Posur, Sebastian},
       title={Frobenius monoidal functors from ambiadjunctions and their lifts
  to {D}rinfeld centers},
        date={2025},
        ISSN={0001-8708},
     journal={Adv. Math.},
      volume={475},
       pages={Paper No. 110344},
         url={https://doi.org/10.1016/j.aim.2025.110344},
}

\bib{FMS}{article}{
      author={Fischman, D.},
      author={Montgomery, S.},
      author={Schneider, H.-J.},
       title={Frobenius extensions of subalgebras of {H}opf algebras},
        date={1997},
        ISSN={0002-9947},
     journal={Trans. Amer. Math. Soc.},
      volume={349},
      number={12},
       pages={4857\ndash 4895},
         url={https://doi.org/10.1090/S0002-9947-97-01814-X},
}

\bib{Hoch}{article}{
      author={Hochschild, G.},
       title={Cohomology of restricted {L}ie algebras},
        date={1954},
        ISSN={0002-9327},
     journal={Amer. J. Math.},
      volume={76},
       pages={555\ndash 580},
         url={https://doi.org/10.2307/2372701},
}

\bib{HY}{article}{
      author={Hu, Jun},
      author={Zhang, Yinhuo},
       title={The {$\beta$}-character algebra and a commuting pair in {H}opf
  algebras},
        date={2007},
        ISSN={1386-923X},
     journal={Algebr. Represent. Theory},
      volume={10},
      number={5},
       pages={497\ndash 516},
         url={https://doi.org/10.1007/s10468-007-9055-4},
}

\bib{JS}{article}{
      author={Joyal, Andr\'{e}},
      author={Street, Ross},
       title={Tortile {Y}ang-{B}axter operators in tensor categories},
        date={1991},
        ISSN={0022-4049},
     journal={J. Pure Appl. Algebra},
      volume={71},
      number={1},
       pages={43\ndash 51},
         url={https://doi.org/10.1016/0022-4049(91)90039-5},
}

\bib{JY}{article}{
      author={Jaklitsch, David},
      author={Yadav, Harshit},
       title={{$\otimes $}-{F}robenius {F}unctors and {E}xact {M}odule
  {C}ategories},
        date={2026},
        ISSN={1073-7928,1687-0247},
     journal={Int. Math. Res. Not. IMRN},
      number={8},
       pages={Paper No. rnag065},
         url={https://doi.org/10.1093/imrn/rnag065},
}

\bib{Kasch}{article}{
      author={Kasch, Friedrich},
       title={Projektive {F}robenius-{E}rweiterungen},
        date={1960/61},
        ISSN={0371-0165},
     journal={S.-B. Heidelberger Akad. Wiss. Math.-Nat. Kl.},
      volume={1960(61)},
       pages={87\ndash 109},
}

\bib{Kas}{book}{
      author={Kassel, Christian},
       title={Quantum groups},
      series={Graduate Texts in Mathematics},
   publisher={Springer-Verlag, New York},
        date={1995},
      volume={155},
        ISBN={0-387-94370-6},
         url={https://doi.org/10.1007/978-1-4612-0783-2},
}

\bib{KelDoc}{inproceedings}{
      author={Kelly, G.~M.},
       title={Doctrinal adjunction},
        date={1974},
   booktitle={Category {S}eminar ({P}roc. {S}em., {S}ydney, 1972/1973)},
      series={Lecture Notes in Math., Vol. 420},
   publisher={Springer, Berlin},
       pages={257\ndash 280},
}

\bib{KO}{article}{
      author={Kirillov, Alexander, Jr.},
      author={Ostrik, Viktor},
       title={On a {$q$}-analogue of the {M}c{K}ay correspondence and the {ADE}
  classification of {$\germ {sl}_2$} conformal field theories},
        date={2002},
        ISSN={0001-8708},
     journal={Adv. Math.},
      volume={171},
      number={2},
       pages={183\ndash 227},
         url={https://doi.org/10.1006/aima.2002.2072},
}

\bib{LR1}{article}{
      author={Larson, Richard~G.},
      author={Radford, David~E.},
       title={Finite-dimensional cosemisimple {H}opf algebras in characteristic
  {$0$} are semisimple},
        date={1988},
        ISSN={0021-8693},
     journal={J. Algebra},
      volume={117},
      number={2},
       pages={267\ndash 289},
         url={https://doi.org/10.1016/0021-8693(88)90107-X},
}

\bib{LS2}{article}{
      author={Larson, Richard~Gustavus},
      author={Sweedler, Moss~Eisenberg},
       title={An associative orthogonal bilinear form for {H}opf algebras},
        date={1969},
        ISSN={0002-9327},
     journal={Amer. J. Math.},
      volume={91},
       pages={75\ndash 94},
         url={https://doi.org/10.2307/2373270},
}

\bib{LW3}{article}{
      author={Laugwitz, Robert},
      author={Walton, Chelsea},
       title={Constructing non-semisimple modular categories with local
  modules},
        date={2023},
        ISSN={0010-3616},
     journal={Comm. Math. Phys.},
      volume={403},
      number={3},
       pages={1363\ndash 1409},
         url={https://doi.org/10.1007/s00220-023-04824-4},
}

\bib{Lyu}{article}{
      author={Lyubashenko, Volodymyr~V.},
       title={Invariants of {$3$}-manifolds and projective representations of
  mapping class groups via quantum groups at roots of unity},
        date={1995},
        ISSN={0010-3616},
     journal={Comm. Math. Phys.},
      volume={172},
      number={3},
       pages={467\ndash 516},
         url={http://projecteuclid.org/euclid.cmp/1104274312},
}

\bib{Mue}{article}{
      author={M\"{u}ller, Eric},
       title={Some topics on {F}robenius-{L}usztig kernels. {I}, {II}},
        date={1998},
        ISSN={0021-8693},
     journal={J. Algebra},
      volume={206},
      number={2},
       pages={624\ndash 658, 659\ndash 681},
         url={https://doi.org/10.1006/jabr.1997.7364},
}

\bib{ML}{book}{
      author={MacLane, Saunders},
       title={Categories for the working mathematician},
      series={Graduate Texts in Mathematics, Vol. 5},
   publisher={Springer-Verlag, New York-Berlin},
        date={1971},
}

\bib{Maj}{book}{
      author={Majid, Shahn},
       title={Foundations of quantum group theory},
   publisher={Cambridge University Press, Cambridge},
        date={2000},
        ISBN={0-521-64868-8},
         url={http://dx.doi.org/10.1017/CBO9780511613104},
        note={Paperback Edition, originally published 1995},
}

\bib{Maj2}{inproceedings}{
      author={Majid, Shahn},
       title={Representations, duals and quantum doubles of monoidal
  categories},
        date={1991},
   booktitle={Proceedings of the {W}inter {S}chool on {G}eometry and {P}hysics
  ({S}rn\'{\i}, 1990)},
       pages={197\ndash 206},
}

\bib{Mas}{article}{
      author={Masuoka, Akira},
       title={Semisimple {H}opf algebras of dimension {$6,8$}},
        date={1995},
        ISSN={0021-2172},
     journal={Israel J. Math.},
      volume={92},
      number={1-3},
       pages={361\ndash 373},
         url={https://doi.org/10.1007/BF02762089},
}

\bib{MN}{article}{
      author={Menini, Claudia},
      author={N\u{a}st\u{a}sescu, Constantin},
       title={When are induction and coinduction functors isomorphic?},
        date={1994},
        ISSN={1370-1444},
     journal={Bull. Belg. Math. Soc. Simon Stevin},
      volume={1},
      number={4},
       pages={521\ndash 558},
         url={http://projecteuclid.org/euclid.bbms/1103408608},
}

\bib{Mon}{book}{
      author={Montgomery, Susan},
       title={Hopf algebras and their actions on rings},
      series={CBMS Regional Conference Series in Mathematics},
   publisher={Published for the Conference Board of the Mathematical Sciences,
  Washington, DC; by the American Mathematical Society, Providence, RI},
        date={1993},
      volume={82},
        ISBN={0-8218-0738-2},
         url={https://doi.org/10.1090/cbms/082},
}

\bib{NT}{article}{
      author={Nakayama, Tadasi},
      author={Tsuzuku, Tosiro},
       title={On {F}robenius extensions. {I}},
        date={1960},
        ISSN={0027-7630},
     journal={Nagoya Math. J.},
      volume={17},
       pages={89\ndash 110},
         url={http://projecteuclid.org/euclid.nmj/1118800455},
}

\bib{NZ}{article}{
      author={Nichols, Warren~D.},
      author={Zoeller, M.~Bettina},
       title={A {H}opf algebra freeness theorem},
        date={1989},
        ISSN={0002-9327},
     journal={Amer. J. Math.},
      volume={111},
      number={2},
       pages={381\ndash 385},
         url={https://doi.org/10.2307/2374514},
}

\bib{Rad}{book}{
      author={Radford, David~E.},
       title={Hopf algebras},
      series={Series on Knots and Everything},
   publisher={World Scientific Publishing Co. Pte. Ltd., Hackensack, NJ},
        date={2012},
      volume={49},
        ISBN={978-981-4335-99-7; 981-4335-99-1},
}

\bib{Rad3}{article}{
      author={Radford, David~E.},
       title={Pointed {H}opf algebras are free over {H}opf subalgebras},
        date={1977},
        ISSN={0021-8693},
     journal={J. Algebra},
      volume={45},
      number={2},
       pages={266\ndash 273},
         url={https://doi.org/10.1016/0021-8693(77)90326-X},
}

\bib{Sch}{article}{
      author={Schneider, Hans-J\"{u}rgen},
       title={Normal basis and transitivity of crossed products for {H}opf
  algebras},
        date={1992},
        ISSN={0021-8693},
     journal={J. Algebra},
      volume={152},
      number={2},
       pages={289\ndash 312},
         url={https://doi.org/10.1016/0021-8693(92)90034-J},
}

\bib{ShiU}{article}{
      author={Shimizu, Kenichi},
       title={On unimodular finite tensor categories},
        date={2017},
        ISSN={1073-7928},
     journal={Int. Math. Res. Not. IMRN},
      number={1},
       pages={277\ndash 322},
         url={https://doi.org/10.1093/imrn/rnv394},
}

\bib{Shi5}{article}{
      author={Shimizu, Kenichi},
       title={The relative modular object and {F}robenius extensions of finite
  {H}opf algebras},
        date={2017},
        ISSN={0021-8693},
     journal={J. Algebra},
      volume={471},
       pages={75\ndash 112},
         url={https://doi.org/10.1016/j.jalgebra.2016.09.017},
}

\bib{Shi}{article}{
      author={Shi, Yuxing},
       title={Finite-dimensional {H}opf algebras over the {K}ac-{P}aljutkin
  algebra {$H_8$}},
        date={2019},
        ISSN={0041-6932},
     journal={Rev. Un. Mat. Argentina},
      volume={60},
      number={1},
       pages={265\ndash 298},
         url={https://doi.org/10.33044/revuma.v60n1a17},
}

\bib{Shi4}{article}{
      author={Shimizu, Kenichi},
       title={Ribbon structures of the {D}rinfeld center of a finite tensor
  category},
        date={2023},
        ISSN={0386-5991},
     journal={Kodai Math. J.},
      volume={46},
      number={1},
       pages={75\ndash 114},
         url={https://doi.org/10.2996/kmj46106},
}

\bib{Tak}{article}{
      author={Takeuchi, Mitsuhiro},
       title={A correspondence between {H}opf ideals and sub-{H}opf algebras},
        date={1972},
        ISSN={0025-2611},
     journal={Manuscripta Math.},
      volume={7},
       pages={251\ndash 270},
         url={https://doi.org/10.1007/BF01579722},
}

\bib{Yad}{article}{
      author={Yadav, Harshit},
       title={Frobenius monoidal functors from (co){H}opf adjunctions},
        date={2024},
        ISSN={0002-9939},
     journal={Proc. Amer. Math. Soc.},
      volume={152},
      number={2},
       pages={471\ndash 487},
         url={https://doi.org/10.1090/proc/16494},
}

\bib{Yet}{article}{
      author={Yetter, David~N.},
       title={Quantum groups and representations of monoidal categories},
        date={1990},
        ISSN={0305-0041},
     journal={Math. Proc. Cambridge Philos. Soc.},
      volume={108},
      number={2},
       pages={261\ndash 290},
         url={https://doi.org/10.1017/S0305004100069139},
}

\end{biblist}
\end{bibdiv}
\bibliographystyle{amsrefs}%

\end{document}